\documentclass{lms}

\usepackage{amssymb,latexsym, amscd}
\usepackage[square, numbers]{natbib}
\usepackage[all]{xy}
\usepackage{graphicx}
\usepackage{mathrsfs}
\usepackage{amsmath}

 \newtheorem{theorem}{Theorem}[section]
 \newtheorem{cor}[theorem]{Corollary}

 \newtheorem{lemma}[theorem]{Lemma}
 \newtheorem{proposition}[theorem]{Proposition}
 \newtheorem{definition}[theorem]{Definition}

 \numberwithin{equation}{section}

\newcommand{\ben}{\begin{equation}}
\newcommand{\een}{\end{equation}}

\newcommand{\integer}{\ensuremath{{\mathbb Z}}}

\newcommand{\real}{\ensuremath{{\mathbb R}}}

\newcommand{\rational}{\ensuremath{{\mathbb Q}}}

\newcommand{\Pp}{\ensuremath{{\mathbb P}}}
\newcommand{\Pf}{\ensuremath{{\mathfrak P}}}

\newcommand{\Aa}{{\mathcal A}}
\newcommand{\DD}{{\mathcal D}}

\newcommand{\PP}{{\mathcal P}}

\newcommand{\BB}{{\mathcal B}}
\newcommand{\KK}{{\mathcal K}}
\newcommand{\ZZ}{{\mathcal Z}}

\newcommand{\LL}{\mathcal{L}}

\newcommand{\Hom}{\mathrm{Hom}}
\newcommand{\Ext}{\mathrm{Ext}}
\newcommand{\Tor}{\mathrm{Tor}}

\newcommand{\Loop}{\mathsf{L}}

\newcommand{\To}{\longrightarrow}

\newcommand{\fiberprod}[2]{\: {}_{#1}  \! \times_{#2}}

\newcommand{\gr}{\mathfrak}

\newcommand{\HHom}{\ensuremath{\mathcal{H}om}}

\newcommand{\RHHom}{\ensuremath{\mathcal{RH}om}}
\newcommand{\lineB}{\ensuremath{\overline{B}}}

\newcommand{\Aut}{\ensuremath{{\mathrm{Aut}}}}

\title[Hochschild cohomology and string topology]{Hochschild cohomology and string topology of global quotient orbifolds}

\author{Andr\'es \'Angel, Erik Backelin and Bernardo Uribe}

\extraline{The third author acknowledges the partial support
provided from the Leibniz prize of Prof. Dr. Wolfgang L\"uck }

\classno{ 55P50, 57R18 (primary), 18G55, 18E30 (secondary)}

\begin{document}
\maketitle

\begin{abstract}
Let $M$ be a connected, simply connected, closed  and oriented manifold,
and $G$ a finite group acting on $M$ by orientation preserving diffeomorphisms.
 In this paper we show an explicit ring isomorphism between the orbifold string topology
 of the orbifold $[M/G]$ and the Hochschild cohomology of the dg-ring obtained by performing
  the smash product between the group $G$ and the singular cochain complex of $M$.
\end{abstract}

\section{Introduction}
String topology stands for the study of the topological properties
associated to the space of smooth  free loops $\LL M$ on a closed
and oriented manifold $M$ of dimension $d$.

The starting point of string topology was the paper
\citep{ChasSullivan} by Chas and Sullivan where the authors
discovered an intersection product in the homology of the free
loop space $$H_p(\LL M) \otimes H_q(\LL M)  \to H_{p+q-d}(\LL
M),$$ having total degree $-d$, which together with the degree 1
operator $H_*(\LL M) \to H_{*+1}(\LL M)$ induced by the circle
action on the loops, endowed the homology of the free loop space with the structure of
a Batalin-Vilkovisky algebra.

Cohen and Jones \citep{CohenJones} developed a homotopical
theoretic realization of string topology, by endowing the Thom
spectrum $\LL M^{-TM}$ with the structure of a ring spectrum
$$\LL M^{-TM} \wedge \LL M^{-TM} \to \LL M^{-TM},$$
that allowed them to show that at the level of homology, the
intersection product of Chas and Sullivan can be recovered by the
product in homology induced by the ring spectrum $\LL M^{-TM}$. In
the same paper Cohen and Jones furthermore showed, generalizing
results of Jones \citep{Jones}, that in the case when $M$ is simply
connected, there is a ring isomorphism between the homology of the
ring spectrum $\LL M^{-TM}$ and the Hochschild cohomology of the
singular cochains of $M$
$$H_*(\LL M^{-TM}) \cong HH^*(C^*(M),C^*(M)).$$

In the case of global quotient orbifolds of the form $[M/G]$ for
$G$ finite group, Lupercio and the third author
\citep{LupercioUribe} constructed the loop groupoid $[P_GM/G]$ as
the natural free loop space of the orbifold, whose homotopic
quotient turns out to be homotopy equivalent to the space of free
loops of the homotopy quotient $M\times_G EG$, i.e.
$$P_GM \times_G EG \simeq \LL(M\times_GEG);$$
  and with this equivalence at hand, Lupercio, Xicot\'encatl and the third author
   \citep{LUX} showed that the homology of the free loop space of the orbifold
$$H_*(\LL(M\times_G EG);\rational)\cong H_*(P_GM ;\rational)^G$$
could also be endowed with the structure of a Batalin-Vilkovisky
algebra; the authors coined this  structure with the name  {\em
orbifold string topology}.

But,  can the orbifold string topology be defined over the
homology with integer coefficients? And, is there any relation
between the orbifold string topology ring and the Hochschild
cohomology of some specific dg-ring? This paper is devoted to
positively answer these two questions.

Let us start with a brief description of the answer of the second
question, as its solution leads the way to solve the first. From
the isomorphism showed by Cohen and Jones between the string
topology ring of a manifold and the Hochschild cohomology of the
singular cochains, one is tempted to try to show that the orbifold
string ring should be isomorphic to the ring
$$HH^*(C^*(M\times_G EG), C^*(M \times_G EG));$$
but due to convergence issues (the Eilenberg-Moore spectral
sequence does not converge in general \citep{Dwyer} ), one cannot
show using standard cosimplicial methods that indeed this ring
recovers the orbifold string ring, and even worse, in some cases we
show (see section \ref{Morita}) that this ring does not
give the appropriate ring structure on the homology of the free loop
space of $[M/G]$. Instead we consider the dg-ring
$C^*(M) \#G$ defined as the smash product of $G$ with the singular
cochains $C^*(M) = C^*(M; \integer)$ of $M$, and we compare its
Hochschild cohomology with the homology ring of the ring spectrum
$P_GM^{-TM}$ that was constructed in \citep{LUX}.

In the case that $M$ is simply connected and connected we find
that that the orbifold string topology ring can be recovered as
the Hochschild cohomology ring of $C^*(M;\rational) \#G$, i.e.
there is an isomorphism of rings
$$HH^*(C^*(M;\rational) \# G, C^*(M; \rational) \#G) \cong H_*(P_GM^{-TM};\rational)^G.$$

This isomorphism is obtained by carefully decomposing the
Hochschild cohomology ring into smaller parts, which leads to the
ring isomorphism
$$HH^*(C^*(M) \# G, C^*(M); \#G) \cong \Ext^*_{\integer G} (\integer, C_*(P_GM^{-TM})).$$
With the previous isomorphism at hand, it was clear that in order
to get a topological counterpart to the Hochschild cohomology of
$C^*(M) \#G$, it was necessary to introduce some sort of Poincar\'e
dual to the universal principal $G$ bundle $EG$. Fortunately the
spaces $EG$ can be approximated by finite dimensional manifolds
$EG_n$ with free $G$ actions, which together with the S-duality
identification $$C^*(EG_n) \simeq C_{-*}(EG_n^{-TEG_n})$$ allow us to
construct a pro-ring spectrum whose homology
$$H^{pro}_*(\LL(M\times_G EG)^{-T(M\times_G EG)}) := \lim_{{\leftarrow}{n}} H_*((P_GM \times_G EG_n)^{-e_0^*T(M \times_G EG_n)})$$ turned out to be isomorphic to the Hochschild cohomology of $C^*(M) \#G$
$$HH^*( C^*(M) \#G,C^*(M) \#G) \cong H^{\rm pro}_*\left(\LL(M\times_G EG)^{-T(M\times_G EG)}\right).$$
This is the main theorem of the paper (Theorem \ref{main
theorem}).

Because of this last isomorphisms we define the orbifold string
topology ring with integer coefficients  to be the ring
$$H^{\rm pro}_*\left(\LL(M\times_G EG)^{-T(M\times_G EG)}\right),$$
and in this way we answer also the first question.

The first five sections are devoted to show the main theorem
(Theorem \ref{main theorem}). After that we have some simple applications
of the main theorem and in section \ref{Morita} we
explain why Hochschild cohomology is not preserved under groupoid
equivalence.

When we started this project we found that the literature about
Hochschild cohomology of dg-rings was a little disperse, it tended
to either be based on formulas whose categorical meaning were not
properly explained or gave much more sophisticated expositions
than was required for our purposes. Therefore we decided to give a
detailed and elementary description of the homological aspects of
Hochschild cohomology that both describe the concrete formulas
needed to compute things as well as the interpretation of
Hochschild cohomology as Ext-groups in the derived category of
dg-modules. This also clarifies the relationships between
algebraic constructions associated to $C^*(M) \#G$ and topological
constructions that deal with free loop spaces.

The layout of the paper is as follows. In
section \ref{section HH for dg-rings} we consider the smash product of
 a discrete group $G$ with a dg-ring $\Aa$ and we give an explicit
resolution of $\Aa \#G$ as a $\Aa \#G^e$-module, that allows us
to decompose the complex that produces the Hochschild cohomology
of $\Aa \#G$ into the composition of two functors
$$ \RHHom_{\integer G}(\integer, \RHHom_{\Aa^e}(\Aa,\Aa \#G)).$$
In section \ref{section loops quotient} we construct cosimplicial
spaces $\Pp_gM$ whose total spaces realize the orbifold loops
$P_gM$, and with this identification at hand we construct a quasi-isomorphism
$$C^*(M) \stackrel{L}{\otimes}_{C^*(M)^e} C^*(M) \#G \simeq C^*(P_GM).$$
In section \ref{section OST} we construct cosimplicial spectra
$\Pf_gM$ whose total spectra realize the spectra $P_gM^{-TM}$, and
that moreover permits to construct a quasi-isomorphism
$$\HHom_{C^*(M)^e}(B(C^*(M)), C^*(M) \#G) \simeq C_*(P_GM^{-TM})$$
where $B(C^*(M))$ denotes the Bar construction of $C^*(M)$.
And in section \ref{section homotopical} we prove the main
theorem of the paper (Theorem \ref{main theorem}) by constructing
a pro-ring spectrum for the orbifold $[M/G]$ whose homology ring is
isomorphic to the Hochschild cohomology of $C^*(M) \#G$.

In section \ref{Applications} we show some applications of the
main theorem, and in section \ref{Morita} we show why the Hochschild cohomology
is not an invariant of the orbifold by presenting two equivalent groupoids with different Hochschild cohomologies.
We finish with three Appendices,     in section \ref{section
derived} we give the preliminaries on derived categories of
dg-modules over dg-rings and we describe several equivalent ways
on which the Hochschild cohomology of a dg-ring  $\Aa$ can be defined.
Section \ref{Appendix B} is devoted to explain the sign notations that
are used in the Bar construction, and section \ref{Appendix C} is devoted
to show that all the constructions performed in this paper for dg-rings that
are free with respect to $\integer$, can be applied to the dg-ring of singular
 cochains $C^*(M)$ on a manifold which in general is not free over $\integer$.
  This second Appendix puts in solid grounds the results of sections \ref{section OST} and \ref{section homotopical}.

\vspace{0.2cm}

 {\it Acknowledgments.} The authors thank the hospitality of the Max Planck Institut f\"ur Mathematik in Bonn, where part of this paper was carried out. The first author would like to thank the support of the Hausdorff Center for Mathematics in Bonn. The third author would like to thank Prof.
Dr. Wolfgang L\"uck for his support at the final stage of this project.

\section{Hochschild cohomology for the smash product of a group and a dg-ring} \label{section HH for dg-rings}

In this section we will describe how to calculate the Hochschild
cohomology for the smash product of a group and a dg-ring.
 The main construction of this section consists of an
explicit cofibrant replacement from which the results that we
claim follow clearly, and which gives an alternative explanation of
the results of \citep{Sanada} over the Hochschild cohomology of
crossed products over commutative rings.

In this section $G$ is a finitely generated discrete group.

Let $\Aa = (A,d_\Aa)$ be a dg-ring together with a group
homomorphism $\sigma: G \to \Aut_{dg-Rings}(A)$, we write $ga :=
g(a) := \sigma(g)(a)$, for $a \in A$, $g \in G$. We refer to such
an $\Aa$ as a $G$-module dg-ring (this name is a dg-analog of the
term module-algebra.)

\begin{definition}
Let $\Aa$ be a $G$-module dg-ring. The smash product gives a
dg-ring that is a principal object of study in this paper
$$\Aa \# G : = \Aa \otimes_\integer \integer G$$
with multiplication given on generators by
$$(x \otimes g)(y \otimes h) := xg(y) \otimes gh$$
and the differential is $d_\Aa \otimes 1$.
\end{definition}

Note that $\Aa$ is a sub-dg-ring of $\Aa \# G$ by the map
$x\mapsto x \otimes 1_G$, and  $\integer G$ is also a sub-dg-ring
of $\Aa \# G$ via the map $g\mapsto 1_\Aa \otimes g$.

In this section we will construct an explicit cofibrant
replacement for $\Aa \# G$ as a $(\Aa \#G)^e$-module, that
together with an explicit diagonal map, will be the main tools to
show that if we consider $\Aa \# G$ as a $G$-module-dg-ring by the
action $g\cdot (x \otimes h) \mapsto g(x) \otimes g h g^{-1}$,
then we have that
\begin{theorem} \label{theorem decomposition HH} There are isomorphisms of graded groups
$$HH_*((\Aa \#G,\Aa \#G) \cong \Tor^*_{\integer G}(\integer,
\Aa \stackrel{L}{\otimes}_{\Aa^e} \Aa \#G)$$
$$HH^*(\Aa \#G,\Aa \#G) \cong \Ext^*_{\integer G}(\integer,
\RHHom_{\Aa^e}(\Aa,\Aa \#G))$$ defined in an appropriate way such
that the second isomorphism becomes one of graded rings.
\end{theorem}

The explicit ring structure on the right hand side of the second
isomorphism will be explained in the proof of the theorem, and it is further emphasized in sections \ref{subsubsection B(G)} and \ref{subsubsection Hom_A2}. This
ring structure will be of use when we will study the string
topology for orbifolds and its relation with Hochschild cohomology
of the dg-rings of singular cochains.

\subsection{Cofibrant replacement for $\Aa \# G$}

Let us consider the cofibrant replacements constructed via the bar
construction (see section \ref{section bar}) for the dg-ring $\Aa$
as an $\Aa^e$-module and for the ring $\integer G$ as a $\integer
G^e$-module
$$B(\Aa) \stackrel{\epsilon}{\to} \Aa   \ \ \ \ \ \ \ B(\integer G)
\stackrel{}{\to} \integer G.$$

If we consider the isomorphism \begin{eqnarray*}(\integer G)
^{k+2} & \to & (\integer G)
^{k+2}\\
(g_0|g_1| \dots |g_{k+1}) & \mapsto & (g_0|g_0g_1|g_0g_1g_2| \dots
|g_0 \dots g_kg_{k+1})\end{eqnarray*} then by transportation of
structures we can change $B(\integer G)$ by an alternative
cofibrant replacement $\lineB(\integer G)$ of $\integer G$ defined
as follows: as a $\integer$-graded module we have that
$$\lineB(\integer G) = \bigoplus_{k=0}^\infty
(\integer G)^{k+2}[k]$$ the differential becomes
$$\delta(h_0| \dots |h_{k+1}) = \sum_{j=0}^k (-1)^j (h_0| \dots |
\widehat{h_j} | \dots | h_{k+1}),$$ the $\integer G^e$-module
structure is
$$(g \otimes k) (h_0| \dots |h_{k+1}) = (gh_0|gh_1|gh_2| \dots
|gh_{k+1}k)$$ and the $\integer G^e$-module homomorphism
$\bar{\epsilon} : \lineB(\integer G) \to \integer G$ is
$$\bar{\epsilon}(h_0|h_1) = h_1 \ \ \ {\rm{and}} \ \ \
\bar{\epsilon}(h_0| \dots | h_{k+1})=0  \ \mbox{for} \ k >1.$$

Notice that for $\lineB(\integer G)$ the diagonal map defined in
(\ref{diagonal map bar}) becomes
\begin{eqnarray} \label{diagonal map lineBarG}
\lineB(\integer G) & \to & \lineB(\integer G) \otimes_{\integer G}
\lineB(\integer G) \\
(h_0| \dots |h_{k+1}) & \mapsto & \sum_{j=0}^k (h_0 | \dots |h_j
|1) \otimes_{\integer G} (h_j|h_{j+1}| \dots | h_{k+1}). \nonumber
\end{eqnarray}

\begin{lemma}
The tensor product $B(\Aa) \otimes_\integer \lineB(\integer G)$
can be endowed with the structure of a $\Aa \# G^e$-module
structure thus making
$$\epsilon \otimes \bar{\epsilon} : B(\Aa) \otimes_\integer
\lineB(\integer G) \to \Aa \# G$$ a cofibrant replacement for $\Aa
\#G$ as an $\Aa \# G^e$-module.
\end{lemma}

\begin{proof}
Let us denote the elements in $\Aa \# G^e$ by $(a \otimes g | b
\otimes k)$ where $a \otimes g$ and $b \otimes k$ belong to $\Aa
\# G$. The $\Aa \# G^e$-module structure of $B(\Aa)
\otimes_\integer \lineB(\integer G)$ is defined as
$$(a \otimes g | b
\otimes k)\left((x_0 | \dots | x_{k+1}) \otimes (h_0 | \dots
|h_{l+1})\right) = $$ $$ (a g(x_0)| g(x_1)| \dots | g(x_k) |
g(x_{k+1}) gh_{l+1}(b) ) \otimes ( gh_0 | gh_1 \dots | gh_l |
gh_{l+1}k).$$ It is a simple calculation to show that indeed the
previous structure is a $\Aa \#G^e$-module  structure on
$B(\Aa) \otimes \lineB( \integer G)$.

To show that $\epsilon \otimes \bar{\epsilon} : B(\Aa)
\otimes_\integer \lineB(\integer G) \to \Aa \# G$ is a morphism of
$\Aa \# G^e$-modules we need only to concentrate our attention to
the elements in $\Aa^2 \otimes_\integer \integer G^2$ and this follows from the commutativity of the following diagram

$$\xymatrix{
(x_0|x_1) \otimes(h_0,h_1) \ar[d]^{\epsilon \otimes
\bar{\epsilon}} \ar[rrr]^{(a\otimes g | b \otimes k)}& & &
(ag(x_0) | g(x_1) gh_1(b) ) \otimes (gh_0| gh_1k)
\ar[d]^{\epsilon \otimes \bar{\epsilon}} \\
(x_0x_1|h_1) \ar[rrr]^{(a\otimes g | b \otimes k)}& & &
(ag(x_0x_1)gh_1(b) | gh_1k) .}$$

The fact that the map $\epsilon \otimes \bar{\epsilon}$ is a
quasi-isomorphism follows from the facts that $\epsilon$ and
$\bar{\epsilon}$ are quasi-isomorphisms. Now we are left to prove
the cofibrant condition. For it consider filtration defined for $q
\geq 0$
$$F^{2q}:=\bigoplus_{p=0}^q \bigoplus_{j=0}^p \left(\left( \Aa \otimes \Aa^{\otimes j} \otimes
\Aa\right) \otimes (\integer G)^{p-j+2}\right)[p]$$ and
$$F^{2q+1}:= F^{2q} \oplus  \bigoplus_{p=0}^q
 \left(\left( \Aa \otimes W^{p+1} \otimes \Aa \right)\otimes (\integer G)^{q-p+2}\right)[q+1] $$
where $W^{p}$ is the $\Aa^e$-submodule of $\Aa^p$ defined as the kernel of the differential $d : \Aa^p \to \Aa^{p+1}$.

The subquotients of the filtration are isomorphic to the $\Aa \# G$-modules

$$F^{2q+1}/F^{2q} \cong \bigoplus_{p=0}^q
 \left(\left( \Aa \otimes W^{p+1} \otimes \Aa \right)\otimes (\integer G)^{q-p+2}\right)[q+1]$$
 $$F^{2q}/F^{2q-1} \cong \bigoplus_{j=0}^q \left(\left( \Aa \otimes \Aa^{\otimes j}/W^{j} \otimes
\Aa\right) \otimes (\integer G)^{q-j+2}\right)[q]$$ where in both
cases the induced differential is only different from zero on the
components of $\Aa$ on the far left and on the far right. As the
$\integer G^e$-modules $(\integer G)^{q-p+2}$ are all free, it
follows that the subquotients $F^{2q+1}/F^{2q}, F^{2q}/F^{2q-1}$
are summands of a direct sum of shifted copies of $\Aa \# G^e$,
therefore $B(\Aa) \otimes_\integer \lineB(\integer G)$ is
cofibrant.

\end{proof}

Now let us define a diagonal map for $B(\Aa) \otimes_\integer
\lineB(\integer G)$:
\begin{align} B(\Aa) \otimes_\integer
\lineB(\integer G) &\to B(\Aa) \otimes_\integer \lineB(\integer G)
\otimes_{\Aa \# G} B(\Aa)
\otimes_\integer \lineB(\integer G)  \label{diagonal B(A)xB(ZG)}\\
(x_0 | \dots | x_{k+1}) \otimes (h_0| \dots | h_{l+1}) &\mapsto  \nonumber \\
\sum_{i=0}^k \sum_{j=0}^l (x_0| ... |x_i
|1) \otimes (h_0| ... |  h_j  | & 1)  \otimes_{\Aa \#G}
(1|x_{i+1}|...  | x_{k+1}) \otimes (h_j|h_{j+1}| ... |h_{l+1})
\nonumber
\end{align}
that is just the juxtaposition of the diagonals maps for $B(\Aa)$
and $\lineB(\integer G)$ defined in (\ref{diagonal map bar}) and
(\ref{diagonal map lineBarG}) respectively. It follows then that the
diagonal map for $B(\Aa) \otimes_\integer \lineB(\integer G)$
satisfies the hypothesis for  Proposition \ref{proposition
diagonal map} and therefore it will induce the ring structure on
the Hochschild cohomology of $\Aa \# G$.

We are now ready to prove the main theorem of this section.

\subsection{ Proof of Theorem \ref{theorem decomposition HH}} We
have that $B(\Aa) \otimes_\integer \lineB(\integer G) \to \Aa\#G$
is a cofibrant replacement  for $\Aa \# G$ as an $\Aa \#
G^e$-module. Therefore we have the isomorphisms
$$HH_*(\Aa \#G , \Aa \#G) =
H^*\left( \left( B(\Aa) \otimes \lineB(\integer G) \right)
\otimes_{\Aa \#G ^e} \Aa \# G \right)$$

and

$$HH^*(\Aa \#G , \Aa \#G) =
H^* \HHom_{\Aa \#G ^e} \left( B(\Aa) \otimes \lineB(\integer G) ,
\Aa \# G \right). $$

Let us consider the sub-dg-rings $\integer G^e \subset \Aa \# G^e$
with  $g \otimes k \mapsto (1 \otimes g | 1 \otimes k)$ and $\Aa^e
\subset \Aa \# G^e$ with $x_0 \otimes x_1 \mapsto (x_0 \otimes 1|
x_1 \otimes 1)$, and note that as such $\Aa^e$ and $\integer G^e$
generate the dg-ring $\Aa \# G^e$. The sub-dg-ring $\Aa^e$ acts
trivially on the component $\lineB(\integer G)$ of $B(\Aa) \otimes
\lineB(\integer G)$ therefore we have an isomorphism

$$\left( B(\Aa) \otimes \lineB(\integer G) \right)
\otimes_{\Aa \#G ^e} \Aa \# G \cong \lineB(\integer G)
\otimes_{\integer G^e} \left( B(\Aa) \otimes_{\Aa^e} \Aa \#G
\right)$$ where the induced action of $g \otimes k \in \integer
G^e $ into $B(\Aa)$ is diagonal on $g $ and trivial on $k$, i.e.
$$(g \otimes k)(x_0 | \dots |x_{k+1}) \mapsto (g(x_0)| \dots |
g(x_{k+1})).$$

The previous argument applies also for the $\HHom$ functor and
therefore we have
$$\HHom_{\Aa \#G^e} \left( B(\Aa) \otimes \lineB(\integer G) ,
\Aa \# G \right) \cong \HHom_{\integer G^e} \left(\lineB(\integer
G) , \HHom_{\Aa^e}( B(\Aa), \Aa\#G)\right).$$

Let us now see more carefully the functors $$\lineB(\integer G)
\otimes_{\integer G^e} \_  \ \ \ \ \ \mbox{and} \ \ \ \
\HHom_{\integer G^e} \left(\lineB(\integer G) , \_ \right).$$ Note
that the action of the elements of the form $(1 \otimes k)$ is by
multiplication on the right, and this action is free on
$\lineB(\integer G)$. Note also that we could generate the ring
$\integer G^e$ by the subrings $1 \otimes \integer G$ and the
image of the diagonal homomorphism $\Delta :\integer G \to
\integer G^e$, $\Delta(g) = g \otimes g^{-1}$. Therefore we could
restrict our attention to the sub-complex $\lineB_G(\integer)$ of
$\lineB(\integer G)$
$$\lineB_G(\integer) = \{(h_0 | \dots | h_{k+1}) \in
\lineB(\integer G) | h_{k+1}=1 \}$$ consisting of the elements
which end in $1$, disregarding the module structure of the subring
$1 \otimes \integer G$ and considering it as a $\integer G$-module
given by the action induced by its image under the diagonal map in
$\integer G^e$.

Note that the $\integer G$ module structure  $\lineB_G(\integer)$
becomes a diagonal action
\begin{eqnarray*}
g \cdot (h_0| \dots |h_k|1) & := & (g \otimes g^{-1})(h_0| \dots
h_k|1) \\
& = & (gh_0 | \dots |gh_k| g1g^{-1})\\
& = &(gh_0 | \dots |gh_k| 1)
\end{eqnarray*}
and that $\lineB_G(\integer)$ becomes a cofibrant replacement for
$\integer$ as a trivial $\integer G$-module.

Thus we have the isomorphisms of complexes
$$
\lineB(\integer G) \otimes_{\integer G^e} \left( B(\Aa)
\otimes_{\Aa^e} \Aa \#G \right)   \cong  \lineB_G(\integer)
\otimes_{\integer G} \left( B(\Aa) \otimes_{\Aa^e} \Aa \#G
\right)$$
$$
\HHom_{\integer G^e} \left(\lineB(\integer G) , \HHom_{\Aa^e}(
B(\Aa), \Aa\#G)\right) \cong  \HHom_{\integer G}
\left(\lineB_G(\integer) , \HHom_{\Aa^e}( B(\Aa), \Aa\#G)\right)
$$
where the $\integer G$-module structure of $B(\Aa)$ is given by
the diagonal action and the $\integer G$-module structure of $\Aa
\#G$ is given by the natural action on $\Aa$ and by conjugation on
$G$, i.e.
\begin{equation*}
g \cdot (x \otimes k)  = (1 \otimes g)(x \otimes k) ( 1 \otimes
g^{-1}) = (g(x) \otimes g k g^{-1}).
\end{equation*}
We can therefore endow $\HHom_{\Aa^e}( B(\Aa), \Aa\#G)$ with the
structure of a $\integer G$-module as follows: for $f \in
\HHom_{\Aa^e}( B(\Aa), \Aa\#G)$  and $g \in G$ we have \begin{eqnarray} \label{G-module st Hom(BA, AG)}
(g f)
(x_0|\dots |x_{k+1}) := g^{-1}(f (g(x_0)|\dots |g(x_{k+1}))) .\end{eqnarray}

Hence we can conclude that there are isomorphisms
\begin{eqnarray*}
HH_*(\Aa \#G,\Aa \#G) & = & H^* \left( \lineB_G(\integer)
\otimes_{\integer G} \left( B(\Aa) \otimes_{\Aa^e} \Aa \#G \right)
\right) \\
& = & \Tor_{\integer G}^*(\integer, \Aa
\stackrel{L}{\otimes}_{\Aa^e}\Aa \#G)
\end{eqnarray*}
\begin{eqnarray*}
HH^*(\Aa \#G,\Aa \#G) & = & H^* \HHom_{\integer G}
\left(\lineB_G(\integer) , \HHom_{\Aa^e}( B(\Aa), \Aa\#G)\right)
\\
& = & \Ext_{\integer G}^*( \integer, \RHHom_{\Aa^e}(\Aa,\Aa \#G)).
\end{eqnarray*}

 \subsubsection{} We are now left to describe the
dg-ring structure of
$$\HHom_{\integer G}
\left(\lineB_G(\integer) , \HHom_{\Aa^e}( B(\Aa), \Aa\#G)\right)$$
that will induce the ring structure on the Hochschild cohomology.
The dg-ring structure is induced from the diagonal map for $B(\Aa)
\otimes_\integer \lineB(\integer G)$ that was defined in
(\ref{diagonal B(A)xB(ZG)}) and from this it follows that for
$$\phi, \psi \in \HHom_{\integer G}
\left(\lineB_G(\integer) , \HHom_{\Aa^e}( B(\Aa), \Aa\#G)\right)$$
we have that
\begin{align} \label{product in HomGHomA2}
 \phi \cdot \psi (h_0| \dots
|h_k|1)(x_0| \dots  | &x_{l+1}) =  \\
 \sum_{i=0}^k \sum_{j=0}^l \left( \phi(h_0| \dots |h_i|1)(x_0 |
\dots |x_j|1) \right) & \left( \psi(h_i| \dots | h_k|1)(1 |
x_{j+1} | \dots | x_{l+1}) \right). \nonumber \end{align}

\noindent This ends the proof of Theorem \ref{theorem
decomposition HH}.

Note that the $\integer G$-module structure on $\Aa \#G$ is given
by $g \cdot (x \otimes k) = (g(a), gkg^{-1})$. Therefore we can
split $\Aa \# G$ into $\integer G $-modules in the following way
\begin{eqnarray} \label{definition A_g}
\Aa \# G \cong \bigoplus_{T \in [G]} \Aa_T \ \ \ \ \ \mbox{with} \
\ \ \ \ \Aa_T=  \bigoplus_{h \in T} \Aa_h \end{eqnarray} where
$[G]$ is the set of conjugacy classes of elements in $G$ and
$\Aa_h$ is the subset of $\Aa \#G$ of elements of the form $x
\otimes h$,
\begin{eqnarray} \label{definition Aa_h}
\Aa_h : = \{ x \otimes h \in \Aa \# G | x \in \Aa\}.
\end{eqnarray}
 Let us choose one element of $G$ for each conjugacy
class in $[G]$, and let us denote this set of representatives by
$\langle G \rangle$. Then we have
\begin{cor} \label{corollary conjugacy classes}
There are isomomorphisms of graded groups
$$HH_*(\Aa \#G , \Aa \#G) \cong \bigoplus_{g \in \langle G \rangle} \Tor_{\integer C_g} (\integer, B(\Aa) {\otimes}_{\Aa^e} \Aa_g)$$
$$HH^*(\Aa \#G , \Aa \#G) \cong \bigoplus_{g \in \langle G \rangle} \Ext_{\integer C_g} (\integer, \HHom_{\Aa^e}( B(\Aa), \Aa_g))$$
where $C_g$ is the centralizer of $g$ in $G$ and $\Aa_g$ is viewed
as a $\integer C_g$ module.
\end{cor}
\begin{proof}
Both isomorphisms are proved by the same argument; let us prove
the second one. As a $\integer G$-module have the isomorphism
$$\HHom_{\Aa^e}( B(\Aa), \Aa \# G) \cong \bigoplus_{T \in [G]} \HHom_{\Aa^e}( B(\Aa), \Aa_T)$$
and therefore we have that
\begin{eqnarray*}
\Ext_{\integer G}^*( \integer, \HHom_{\Aa^e}(B(\Aa), \Aa \#G)) & \cong & \bigoplus_{T \in [G]} \Ext_{\integer G}^*( \integer, \HHom_{\Aa^e}(B(\Aa), \Aa_T))\\
& \cong & \bigoplus_{g \in \langle G \rangle } \Ext_{\integer
C_g}^*( \integer, \HHom_{\Aa^e}(B(\Aa), \Aa_g)).
\end{eqnarray*}

\end{proof}

\begin{cor}
If $\Aa$ is a free $\integer G$-module-dg-ring, then we have
isomorphisms of graded groups
$$HH_*(\Aa \#G , \Aa \#G) \cong \bigoplus_{g \in \langle G \rangle} H^* \left( \left( B(\Aa) {\otimes}_{\Aa^e} \Aa_g \right)^{C_g} \right)$$
$$HH^*(\Aa \#G , \Aa \#G) \cong \bigoplus_{g \in \langle G \rangle}  H^* \left( \left( \HHom_{\Aa^e}( B(\Aa), \Aa_g) \right)^{C_g} \right).$$
\end{cor}
\begin{proof}
The isomorphisms follow form Corollary \ref{corollary conjugacy
classes} and the fact that for free modules the invariants and the
coinvariants are isomorphic.
\end{proof}

\subsection{Further results} We end up this section with some results that
follow from the proof of Theorem \ref{theorem decomposition HH}.

\subsubsection{} \label{subsubsection B(G)} The diagonal map of $B(\Aa)$ defined on (\ref{diagonal map
bar}) endows the complex $$\HHom_{\Aa^e}( B(\Aa), \Aa\#G)$$ with
the structure of a dg-ring which is moreover a $\integer G$-module
dg-ring. Following the notation of (\ref{definition Aa_h}) if we take two functions
$$ \phi \in \HHom_{\Aa^e}(\Aa^{k+2}[k], \Aa_g) ,\ \ \ \psi \in  \HHom_{\Aa^e}(\Aa^{l+2}[l], \Aa_h)$$
then its product  $\phi \cdot \psi$ lives in $\HHom_{\Aa^e}(\Aa^{k +l+2}[k+l], \Aa_{gh})$ and is defined as
\begin{align*}
(\phi \cdot \psi)(a_0 |   \dots |  & a_{k+l+2}) \\
=& (-1)^{|\psi| \varepsilon_k} (\phi_g(a_0|\dots |a_k|1)  \otimes g )(\psi_h(1|a_{k+1}| \dots |a_{k+l+1}) \otimes h) \\
= & (-1)^{|\psi| \varepsilon_k} \phi_g(a_0|\dots |a_k|1) g(\psi_h(1|a_{k+1}| \dots |a_{k+l+1}) )\otimes gh
\end{align*}
where we use the convention  $$\phi({\bf a}) = \sum_g (\phi_g({\bf a}) \otimes g).$$
The $\integer G$-module structure of $\HHom_{\Aa^e}( B(\Aa), \Aa\#G)$ is defined in ({\ref{G-module st Hom(BA, AG)}) and it is easy to check that
for $k \in G$ we $(k\phi)(k\psi) = k(\phi \cdot \psi)$.

We therefore have that $\Ext^*_{\Aa^e}(\Aa,\Aa \# G)$ becomes a
$\integer G$-module ring.

\subsubsection{} \label{subsubsection Hom_A2} Whenever $W$ is a $\integer G$-module ring, the complex
$\HHom_{\integer G}(\lineB_G(\integer), W)$ can be endowed with
the structure of a dg-ring with a product structure induced by the
formula in (\ref{product in HomGHomA2}), i.e. for $\alpha, \beta \in
\HHom_{\integer G}(\lineB_G(\integer), W)$ we have that
$$\alpha \cdot \beta \ (h_0 | \dots | h_l|1) = \sum_{j=0}^l
\alpha(h_0| \dots | h_j | 1) \beta (h_j| \dots |h_l|1).$$ This
dg-ring structure  makes $\Ext^*_{\integer G}(\integer, W)$ into a
ring. In the case that $W = \integer$, the ring $\Ext^*_{\integer
G}(\integer, \integer)$ is isomorphic to the ring $H^*(BG,
\integer)$.

\subsubsection{} If we filter both complexes
$$\lineB_G(\integer)
\otimes_{\integer G} \left( B(\Aa) \otimes_{\Aa^e} \Aa \#G
\right)$$
$$\HHom_{\integer G}
\left(\lineB_G(\integer) , \HHom_{\Aa^e}( B(\Aa), \Aa\#G)\right)$$
by the degree of the elements in $\lineB_G(\integer)$ then we get
spectral sequences which abut to the Hochschild homology and
cohomology respectively and whose second page is given by
$$ E^2=\Tor_{\integer G}^*(\integer, \Tor_{\Aa^e}(\Aa, \Aa \# G))
\Rightarrow HH_*(\Aa \# G, \Aa \#G)$$ and
$$ E^2=\Ext_{\integer G}^*(\integer, \Ext_{\Aa^e}(\Aa, \Aa \# G))
 \Rightarrow HH^*(\Aa \# G, \Aa \#G).$$

In the case of the Hochschild cohomology ring, the spectral
sequence is a sequence of rings where
$$ \Ext_{\integer G}^*(\integer, \Ext_{\Aa^e}(\Aa, \Aa \# G))$$
has the ring structure explained in \ref{subsubsection B(G)} and
\ref{subsubsection Hom_A2}.

\section{Loops on quotient spaces} \label{section loops quotient}
This section is devoted to show the relation between the complexes
$$B(\Aa) \otimes_{\Aa^e} \Aa \#G \ \ \ \ \ \ \mbox{and}  \ \ \  \ \
\HHom_{\Aa^e}(B(\Aa),\Aa\#G),$$ and the cochains and chains
respectively for the loops on the groupoid $[X/G]$ whenever we
take $\Aa= C^*(X)$.

\subsection{Loops on $[X/G]$}

Let $X$ be a connected CW-complex of finite type and $G$ a
discrete group acting on $X$. Denote by $[X/G]$ the action
groupoid whose objects are $X$ and whose morphisms are $X \times
G$ with $s(x,g)=x$ and $t(x,g)=xg$.

The loops on $[X/G]$ can be understood as the groupoid whose
objects are the functors from the groupoid $[\real/\integer]$ to
$[X/G]$ and whose morphisms are given by natural transformations.
More explicitly, we have

\begin{definition} \label{definition loop groupoid}
The loop groupoid $\Loop[X/G]$ for $[X/G]$ is the action groupoid
$[P_GX/G]$ whose space of objects is
$$P_GX := \bigsqcup_{g \in G} P_gX \times \{g\} \ \ \ \
\mbox{with} \ \ \ \  P_gX := \{f : [0,1] \to X | f(0)g =
f(1)\}$$ and which are endowed with the $G$-action
\begin{eqnarray*}
G \times P_GX & \to & P_GX\\
((f,k),g) & \mapsto & (fg, g^{-1}kg).
\end{eqnarray*}
\end{definition}

In Theorem 2.3 of \citep{LUX} it is proven that for $G$ a finite
group there exist a natural weak homotopy equivalence
$$ \LL(EG \times_G X) \to EG \times_G P_GX$$
between the free loop space of the homotopy quotient $EG \times_G
X$ and the homotopy quotient of the loop groupoid. This proof can
be easily generalized to the case that $G$ is discrete.

In this section we will consider $P_GX$ as a $G$-space and will
not work with its homotopy quotient. Now we will give a
cosimplicial description for the spaces $P_gX$.

\subsection{Cosimplicial description for $P_gX$}

Let us start by recalling a cosimplicial construction of the space
of paths of a topological space $X$ that was done in \citep{Jones}
(we will use the properties of simplicial and cosimplicial spaces
that are developed in \citep{SegalCategories} and in
\citep{BottSegal} respectively). Take the category $\Delta$ whose
objects are the finite ordered sets ${\gr n} = \{0, 1, \dots, n\}$
and whose morphisms $\Delta({\gr n}, {\gr m})$ are the order
preserving maps $s \colon {\gr n} \to {\gr m}$. The morphisms of
$\Delta$ are generated by:
\begin{itemize}
\item The face maps $\delta_i \in \Delta({\gr n}-1,{\gr n})$, $0
\leq i \leq n$; the unique order preserving map  whose image does
not contain $i$. \item The degeneracy maps $\sigma_i \in
\Delta({\gr n}+1, {\gr n})$, $0 \leq i \leq n$; the unique
surjective order preserving map which repeats $i$.
\end{itemize}
These generators satisfy the usual cosimplicial relations. We
define a simplicial object in a category ${\gr C}$ to be a
contravariant functor $\Delta \to {\gr C}$ and a cosimplicial
object in ${\gr C}$ to be a covariant functor $\Delta \to {\gr
C}$.

Now let us define the simplicial sets $\lambda^n$ where
$\lambda^n({\gr m}):=\Delta({\gr m},{\gr n})$ with the natural
coface and codegeneracy maps induced by $\Delta$. The geometric
realization $|\lambda^n|$ of the simplicial set $\lambda^n$ is
homeomorphic to the $n$-simplex $\Delta_n$. In particular when
$n=1$, the simplicial space $\lambda^1$  has as geometric
realization the 1-simplex $\Delta_1$. One could think of it as
$\lambda^1({\gr n}) = {\gr n}+2$ with cofaces $\lambda^1({\gr n})
\to \lambda^1({\gr n} -1)$ and codegeneracies $\lambda^1({\gr
n}-1) \to \lambda^1({\gr n})$ given by order preserving maps that
send $0$ to $0$, and $n+2$ to $n+1$ in the former case and $n+1$
to $n+2$ in the latter.

Consider the cosimplicial space $\Pp X:= X^{\lambda^1}$ defined by
taking the maps from the simplicial set $\lambda^1$ to $X$. Then
we have that  $$\Pp X({\gr n}):= \mbox{Map}(\lambda^1({\gr n}),
X)\cong X^{n+2}$$ and the cosimplicial structure maps are the ones
induced by the simplicial structure of $\lambda^1$. We have the
tautology \citep[Prop 5.1]{BottSegal}

\begin{lemma}
There is a natural homeomorphism between the space of paths
$$PX:= \mbox{Map}(|\lambda^1|,X) =
\mbox{Map}(\Delta_1,X)$$ and the total space $|\Pp X|$ of the
cosimplicial space $\Pp X$.
\end{lemma}

Consider $\overline{\Pp}_gX$ the subcosimplicial space of $\Pp X$
defined by the spaces
$$\overline{\Pp}_gX({\gr n}) := \{(x_0, \dots, x_{n+1}) \in \Pp X({\gr n}) \colon x_0g = x_{n+1} \} \cong X^{n+1}.  $$
Clearly the coface and codegeneracy maps are well defined in
$\overline{\Pp}_g X$ and it follows that the total space of
$\overline{\Pp}_g X$ is homeomorphic to the space of paths $\gamma
\in PX$ such that $\gamma(0)g=\gamma(1)$.

Define the cosimplicial space $\Pp_g X$ by dropping the last
coordinate of $\overline{\Pp}_gX$ (as the last coordinate
$x_{n+1}$ on the $n$-the level is equal to $x_0g$) therefore
having
$$\Pp_gX ({\gr n}):= X^{n+1}$$ together with
the induced coface and codegeneracy maps given by:
\begin{eqnarray}
\label{cofaces P_gX} \delta_i(x_0, \dots, x_{n}) & = & (x_0, \dots, x_{i-1}, x_i, x_i, x_{i+1}, \dots, x_{n}) \ \ \ \ 0 \leq i \leq n\\
\nonumber\delta_{n+1}(x_0, \dots, x_{n}) & = & (x_0, \dots, x_{n}, x_0g)\\
\nonumber\sigma_i(x_0, \dots, x_{n}) & = & (x_0, \dots, x_i,
x_{i+2}, \dots, x_{n}) \ \ \ \ 0 \leq i \leq n-1.
\end{eqnarray}
We can now consider the maps
\begin{eqnarray}
\phi_k \colon \Delta_k \times P_gX & \to & X^{k+1}=\Pp_gX({\gr n}) \label{functions phi_k} \\
(t_1, \dots, t_k) \times \gamma & \mapsto & (\gamma(0),
\gamma(t_1), \dots, \gamma(t_k)). \nonumber
\end{eqnarray}
together with $\hat{\phi}_k : P_gX \to \mbox{Map}(\Delta_k, \Pp_gX)$
their adjoints. Let $$\phi: P_gX \to \prod_{k \geq 0}
\mbox{Map}(\Delta_k, X^{k+1})$$ denote the product of the maps
$\hat{\phi}_k$, then we have

\begin{lemma} \label{lemma total map of spectra}
The image of the map $\phi$ is the total space of $\Pp_gX$ and
therefore  $\phi : P_gX \to |\Pp_gX|$ is a homeomorphism.
\end{lemma}

\subsection{Cochains on loops of $[X/G]$}

Let us denote by $${C^*}:=C^*(X;\integer)$$ the dg-ring of
$\integer$-valued singular cochains on $X$. Because $G$ acts on
$X$ then ${C^*}$ is endowed with the structure of a $\integer
G$-module-dg-ring.

For $g \in G$ let ${C^*_g}$ be the submodule of ${C^*}\#G$ generated
by the elements of the form $x \otimes g$ as it was defined in (\ref{definition
A_g}). Note that ${C^*_g}$ inherits the structure of a
${C^*}^e$-module when one takes ${C^*}^e \subset {C^*} \#G^e$ and ${C^*_g}
\subset {C^*}\#G$ in the following way: let $a_0 \otimes b_0 \in
{C^*}^e$ and $x \otimes g \in {C^*_g}$, then
\begin{eqnarray*}
(a_0 \otimes a_1) \cdot x \otimes g & = & (a_0 \otimes 1 | a_1
\otimes 1)
(x\otimes g)  \nonumber\\
&=&(a_0 \otimes 1)(x \otimes g)( a_1 \otimes 1) \nonumber\\
& =& a_0 x g(a_1) \otimes g.
\end{eqnarray*}

\begin{theorem} \label{theorem Tor=H^*(loops)}
For $X$ a connected and simply connected CW-complex of finite
type, there exists a homomorphism of complexes
$$B({C^*}) \otimes_{{C^*}^e} {C^*_g} \stackrel{\simeq}{\To} C^* (P_gX ; \integer)$$
that moreover is a quasi-isomorphism, and therefore induces an
isomorphism
$$\Tor^*_{{C^*}^e}({C^*},{C^*_g}) \stackrel{\cong}{\To} H^* (P_gX ; \integer).$$
\end{theorem}

\begin{proof}
Take the cosimplicial space $\Pp_gX$ an apply to it  the singular
cochains functor; we obtain the simplicial cochain complex
$C^*(\Pp_gX)$ . By Lemma 7.1 of \citep{BottSegal} we have that the
total complex (or its
 realization)  $|C^*(\Pp_gX)|$ of the simplicial cochain complex $C^*(\Pp_gX)$ is quasi-isomorphic
   to $C^*(P_gX)$, as we have that $P_gX$ is homeomorphic to the total space of $\Pp_gX$.
Then we notice that the total complex $|C^*(\Pp_gX)|$ is
quasi-isomorphic
 to the complex $B({C^*}) \otimes_{{C^*}^e} {C^*_g}$ when one applies carefully
 the Eilenberg-Zilber theorem for product spaces \citep{EilenbergZilber}. These two facts together
  imply the theorem. Let us see each one with more detail.

\begin{lemma} \label{lemma total-cochains = cochains-total}
There is a homomorphism of graded complexes
$$|C^*(\Pp_gX) | \To C^*(P_gX)$$
that when $X$ is a connected and simply connected CW-complex of
finite type it becomes a quasi-isomorphism.
\end{lemma}

\begin{proof} The lemma is a direct consequence of Proposition 5.3
and Lemma 7.1 of \citep{BottSegal}. Let us see how the homomorphism is defined.

Recall from section 5 of \citep{BottSegal} that the group of
homogeneous elements of degree $r$ of the total complex
$|C^*(\Pp_gX) |$ is the direct sum of the groups
$$C^{n+r}(\Pp_gX({\gr n})) = C^{n+r}(X^{n+1})=C^*(X^{n+1})[n]^r  \ \ \ \ \ \ \mbox{for} \ \  n\geq 0$$
and therefore
$$ |C^*(\Pp_gX) |=  \bigoplus_{n \geq 0} C^*(X^{n+1})[n].$$

Let us consider the composition of the homomorphisms
\begin{eqnarray*} C^{n+r}(X^{n+1}) \stackrel{\phi_n^*}{\To}
C^{n+r}(\Delta_n \times P_gX ) \stackrel{\int_{\Delta_n}}{\To}
C^r(P_gX)
\end{eqnarray*}
where the functions $\phi_n$ were defined in (\ref{functions phi_k})
and $\int_{\Delta_n}$ evaluates a cochain on the class
$[\Delta_n]$, that is: the composition of the Eilenberg-Zilber map
together with the evaluation on the class $[\Delta_n]$
$$C^{n+r}(\Delta_n \times P_gX ) \To C^n(\Delta_n) \otimes
C^r(P_gX) \To C^r(P_gX).$$

Because the operators $\int_{\Delta_n}$ satisfy the property
$$(-1)^n d \left( \int_{\Delta_n}\alpha \right)= \int_{\Delta_n}
d\alpha - \int_{\partial \Delta_n} \alpha$$ we have that the maps
$$C^*(X^{n+1})[n]^r =C^{n+r}(X^{n+1}) \To C^{r}(P_gX)$$ assemble to
define homomorphism of complexes of degree zero
$$F:|C^*(\Pp_gX) | \To C^*(P_gX).$$ It follows from Proposition 5.3 of
\citep{BottSegal} that $F$ is a quasi-isomorphism as the
connectivity of $X$ ($\geq 2$) is higher than the connectivity of the
simplicial set that defines $\Pp_gX$ (which in this case is the circle).
\end{proof}

\begin{lemma} \label{lemma Tor=total-cochain}
There is a homomorphism of graded complexes
$$B({C^*}) \otimes_{{C^*}^e} {C^*_g} \to |C^*(\Pp_gX) |$$ that is moreover a
quasi-isomorphism.
\end{lemma}
\begin{proof}
Take the isomorphisms
\begin{eqnarray*}
{C^*}^{n+1}  & \stackrel{\cong}{\to} & {C^*}^{n+2}\otimes_{{C^*}^e} {C^*_g}\\
(a_0 | \dots |a_n) & \mapsto & (a_0 | \dots | a_n |1)
\otimes_{{C^*}^e} (1)
\end{eqnarray*}
together with the induced differential
\begin{equation} \label{induced differential}
(a_0| \dots | a_n) \mapsto \sum_{j=0}^{n-1} (-1)^j (a_0 | \dots
|a_ja_{j+1}| \dots |a_n) + (-1)^{n} (g(a_n)a_0| \dots |a_{n-1})
\end{equation}
that comes from the bar differential on $\oplus_{k \geq 0} {C^*}^{k+2}$ (see Appendix A in section \ref{section derived}); note that the last
expression in the formula (\ref{induced differential}) comes from
the equivalences
\begin{eqnarray*}
(a_0 | \dots | a_n ) \otimes_{{C^*}^e} (1) & = & (a_0 | \dots | a_{n-1} |1) \otimes_{{C^*}^e} g(a_n) \\
&=&(g(a_n)a_0 | \dots | a_{n-1} |1) \otimes_{{C^*}^e} (1).
\end{eqnarray*}

Now let us consider the quasi-isomorphisms defined by the
Eilenberg-Zilber map
$${C^*}^{n+1} \stackrel{\simeq}{\to}  C^*(X^{n+1}) =C^* (\Pp_gX( {\gr n}))$$
and note that the induced differential ${C^*}^{n+1} \to {C^*}^n$
coming from the coface maps defined in (\ref{cofaces P_gX}) is the
same as in (\ref{induced differential}).

We can conclude that the compositions of the maps
$${C^*}^{n+2} \otimes_{{C^*}^e} {C^*_g} \stackrel{\cong}{\to} {C^*}^{n+1} \stackrel{\simeq}{\to} C^*(X^{n+1})$$
induce the maps
$${C^*}^{n+2}[n] \otimes_{{C^*}^e} {C^*_g} \stackrel{\cong}{\to} {C^*}^{n+1}[n] \stackrel{\simeq}{\to} C^*(X^{n+1})[n]$$
that induce the desired quasi-isomorphism
$$B({C^*}) \otimes_{{C^*}^e} {C^*_g} \stackrel{\simeq}{\to} |C^*(\Pp_gX)|.$$
\end{proof}

Now we can finish with the proof of Theorem \ref{theorem
Tor=H^*(loops)} by composing the quasi-isomorphisms defined in the
Lemmas \ref{lemma total-cochains = cochains-total} and \ref{lemma
Tor=total-cochain}
$$B({C^*}) \otimes_{{C^*}^e} {C^*_g} \stackrel{\simeq}{\to} |C^*(\Pp_gX)| \stackrel{\simeq}{\to} C^*(P_gX)$$
that moreover induce the isomorphisms
$$\Tor_{{C^*}^e}({C^*},{C^*_g}) \stackrel{\cong}{\to} H^*(|C^*(\Pp_gX)|)  \stackrel{\cong}{\to} H^*(P_gX).$$
\end{proof}

We can assemble all the ${C^*_g}$'s into ${C^*} \#G$ thus obtaining

\begin{cor}
If $X$ is connected and simply connected and $G$ is finite, then
there is a quasi-isomorphism
$$B(C^*(X)) \otimes_{C^*(X)^e} C^*(X) \#G \stackrel{\simeq}{\to} C^*(P_GX)$$
that induces the isomorphism
$$\Tor_{C^*(X)^e}(C^*X, C^*(X) \#G) \cong  H^*(P_GX).$$
\end{cor}
The previous corollary together with Theorem \ref{theorem decomposition HH} implies that
\begin{proposition}
If $X$ is connected and simply connected and $G$ is finite, then
there is an isomorphism of graded groups
$$HH_*(C^*(X) \#G, C^*(X) \#G) \cong \Tor_{\integer G}(\integer, C^*(P_GX)),$$
and whenever $G$ acts freely on $X$ the isomorphism becomes
$$HH_*(C^*(X) \#G, C^*(X) \#G) \cong H^*\left( C^*(P_GX)^G \right).$$
\end{proposition}

\subsection{Chains on loops of $[X/G]$}

From the previous section we have the quasi-isomorphisms
\begin{equation} \label{q.i.'s of C^*P_gX}
B({C^*}) \otimes_{{C^*}^e} {C^*_g} \stackrel{\cong}{\to} \bigoplus_{n
\geq 0} {C^*}^{n+1}[n] \stackrel{\simeq}{\to} |C^*(\Pp_gX)|
\stackrel{\simeq}{\to} C^*(P_gX).\end{equation} If $P_gX$ is of
finite type i.e. the cohomology is finitely generated in each
degree, we can dualize the maps of (\ref{q.i.'s of C^*P_gX}) thus
obtaining quasi-isomorphisms
\begin{equation*}
C_*(P_gX) \stackrel{\simeq}{\to} |C^*(\Pp_gX)|^\vee
\stackrel{\simeq}{\to} \bigoplus_{n \geq 0} {C^*}^{n+1}[n]^\vee
\stackrel{\simeq}{} \bigoplus_{n \geq 0}
\HHom_{\integer}({C^*}^{n}[n], {C^*_g}^\vee)
\end{equation*}
where $T^\vee := \Hom_\integer(T , \integer)$ and the last map is
induced by the isomorphisms $\HHom_{\integer}({C^*}^{n},{C^*}^\vee_g) \to
{{C^*}^{n+1}}^\vee $ with $\phi \mapsto \bar{\phi}$ and
$$\bar{\phi}(a_1| \dots |a_n|b) = \phi(a_1|\dots|a_n)(b).$$

Whenever $X$ is a compact closed oriented manifold of dimension
$l$, there is a homomorphism  of graded groups
$$C^*(X) \to C_{l-*}(X)$$ that induces an isomorphism in cohomology, and that moreover induces
an structure of $C^*(X)^e$-module on $C_{l-*}(X)$. Therefore we could
apply the natural isomorphism
$$\HHom_{\integer}({C^*}^{n},{C^*_g}^\vee) \cong \HHom_{{C^*}^e}({C^*}^{n+2},{C^*_g}^\vee)$$
that yields then that  there is a quasi-isomorphism
$$C_*(P_gX) \stackrel{\simeq}{\to} \HHom_{{C^*}^e}(B({C^*}),{C^*_g}^\vee)$$
that induces the isomorphism
$$H_*(P_gX) \cong H^* \HHom_{{C^*}^e}(B({C^*}),{C^*_g}^\vee).$$

The previous proof is very sketchy and not rigorous at all; we
postpone its formal proof to the next sections.

In the next section we will show explicitly how the singular
chains of $P_gX$ are related to the complex
$\HHom_{{C^*}^e}(B({C^*}),{C^*_g})$, and moreover we will show how the
ring structure of  $H_*(P_GX)$ defined in \citep{LUX} is isomorphic
to the ring  $H^* \HHom_{{C^*}^e}({B(C^*}),{C^*} \#G)$ whenever $G$ is finite
and $X$ is a compact, connected and simply connected oriented
manifold.

\section{String topology for orbifolds}\label{section OST}
This section is devoted to the construction of the  topological counterpart for
 the $G$-module dg-ring
$$\HHom_{{C^*}^e}(B({C^*}),{C^*}\#G)$$
whenever we have a global quotient orbifold $[M/G]$ with $M$ a differentiable, oriented and compact manifold, $G$ finite group acting by orientation preserving diffeomorphisms and ${C^*} = C^*(M, \integer)$.

Section 4.1 generalizes the construction of \citep{Cohen-Atiyahduality} to the equivariant
 case to give a symmetric ring spectra over pointed  $G$-spaces model of the Thom spectrum $M^{-TM}$.

In section \ref{subsection OST} we will start by recalling the construction of the
 orbifold string topology of $[M/G]$ performed in \citep{LUX}, which is based on the union of spectra
$$P_GM^{-TM}:= \bigsqcup_{g \in G}P_gM^{-TM}$$
and the maps
\begin{eqnarray} \label{product spectra}
P_gM^{-TM} \wedge P_hM^{-TM} \to P_{gh}M^{-TM}.
\end{eqnarray}

Then in section \ref{section cosimplicial spectra} we will construct for each $g$
 a cosimplicial spectrum $\Pf_gM$ such that its total spectrum is homeomorphic to $P_gM^{-TM}$.
  With the cosimplicial spectrum $\Pf_gM$ at hand, we will construct a
  quasi-isomorphism between $C_*(P_gM^{-TM})$ and $\HHom_{{C^*}^e}(B({C^*}),{C^*_g})$
  which will assemble into a quasi-isomorphism
$$C_*(P_GM^{-TM}) \stackrel{\simeq}{\To} \HHom_{{C^*}^e}(B({C^*}), {C^*} \#G).$$

In section \ref{section multiplicative structures} we will focus on multiplicative issues.
 We will construct  maps of cosimplicial spectra $$\Pf_gM \wedge \Pf_hM \to \Pf_{gh}M$$
 that will be compatible with the map of (\ref{product spectra}) and we will show that
 after passing to chains, it will be compatible with the ring structure of $$\HHom_{{C^*}^e}(B({C^*}),{C^*} \#G).$$
 We will show that the quasi-isomorphism
 $$C_*(P_GM^{-TM}) \stackrel{\simeq}{\to} \HHom_{{C^*}^e}(B({C^*}), {C^*} \#G)$$
 will be a $G$-equivariant map of $A_\infty$ rings, and
 in particular we will be able show that there is $G$-equivariant isomorphism of rings
$$H_*(P_GM^{-TM}) \cong H^* \HHom_{{C^*}^e}(B({C^*}),{C^*}\#G).$$

\subsection{Multiplicative properties of Atiyah Duality} \label{subsection MPAD}
This section is based on the construction done in \cite{Cohen-Atiyahduality} of a
 symmetric ring spectra model of $M^{-TM}$.  Recall the classical construction of
  the spectrum $M^{-TM}$, given an embedding $e :M \hookrightarrow \mathbb{R}^k$,
  let  $\nu(e) \to M$ the normal bundle of the embedding and denote by $M^{\nu(e)}$
  its Thom construction. Define the spectrum
$$M^{-TM}: = \Sigma^{-k} M^{\nu(e)}$$
as the $k$-th de-suspension of the Thom space $M^{\nu(e)}$.

The Spanier-Whitehed dual of $M_+$ is the function spectrum $F(M,S)$ and this,
 by Atiyah duality, can be identified with $M^{-TM}$. Therefore, the diagonal
 map $\Delta : M \to M \times M$ induces a map of spectra
$$\Delta^*: M^{-TM} \wedge M^{-TM} \to M^{-TM}$$
which makes $M^{-TM}$ into a ring spectrum in the stable homotopy category.
 Following \cite{Cohen-Atiyahduality}, we rigidify this ring spectrum using
 the machinery of symmetric spectra in general model categories of \cite{Hovey_spectraand}
 to take into account the $G$-action.

If $\mathcal{C}$ is a symmetric monoidal model category and $K$ is a cofibrant
object of $\mathcal{C}$, Hovey defined categories $Sp^{\mathbb{N}}(\mathcal{C},K)$
and $Sp^\Sigma(\mathcal{C},K)$ of spectra and symmetric spectra over $\mathcal{C}$.
 A spectrum is a sequence $\{X_n\}$ of elements of $\mathcal{C}$ together
  with morphisms $\epsilon_{n,m} :  X_n \otimes K^{\otimes m} \to X_{n+m}$
  that make $\{X_n\}$ a module over $Sym(K)=(S^0,K,K^{\otimes 2},\cdots,K^{\otimes n},\ldots)$.
   Similarly, symmetric spectra  are symmetric sequences $\{X_n\}$ with
   $\Sigma_n \times \Sigma_m$-equivariant morphisms
   $\epsilon_{n,m} :  X_n \otimes K^{\otimes m} \to X_{n+m}$ that make it a $Sym(K)$-module.

For our purposes, we take $\mathcal{C}$ the model category of based topological
 spaces with $G$-action with the (fine) equivariant model structure and $K=S^V$,
  the one-point compactification of the representation $V$.

Recall from \cite{illman}, that every manifold with the action of a finite group
 admits the structure of a $G-CW$-complex, and therefore $S^V$ is cofibrant
 in the (fine) equivariant model structure.

Given a manifold $M$ with a $G$-action and an equivariant embedding
$ e : M \hookrightarrow V$ into an orthogonal representation of $G$,
 we will define symmetric ring spectra over $\mathcal{C}$, $M^{-\tau}(e)$
  and $F(e)$, together with a map $\alpha : M^{-\tau}(e) \rightarrow F(e)$
   that implements Atiyah duality at the level of symmetric ring spectra.

Let $\nu(e) \subseteq V$ be an equivariant tubular neighborhood of the
embedding $e$. Consider the space $L(V,V^n)^G$ of equivariant linear
 embeddings $\phi: V \rightarrow V^n$ with a trivial $G$-action and
 a $\Sigma_n$-action coming from the permutation action on $V^n$.

For $\phi \in L(V,V^n)^G$, let $\theta(\phi)$ an equivariant tubular
 neighborhood of $\phi \circ e$. Define the sequence of pointed $G$-spaces
$$
\tilde{M}_n^{-\tau}(e)  = \{ (\phi,x ) \mid \phi \in L(V,V^n)^G, x \in V^n / \left ( V^n - \theta(\phi) \right ) \}
$$
A point in $\tilde{M}_n^{-\tau}(e)$ is an equivariant linear
embedding $\phi : V \rightarrow V^n$ together with a point in the
 Thom space of the normal bundle to the embedding $\phi \circ e :  M \hookrightarrow V^n$.
  We have a fiber bundle $\tilde{M}_n^{-\tau}(e) \rightarrow L(V,V^n)^G$
  with fiber the Thom space $V^n/\left ( V^n - \theta(\phi)\right )$.
   The $\Sigma_n$-action on $V^n$ induces action on $\tilde{M}_n^{-\tau}(e)$.

Consider the sequence of pointed $G$-spaces
$$
M_n^{-\tau}(e) = \tilde{M}_n^{-\tau}(e) / \{(\phi, \infty) \mid \phi \in L(V,V^n)^G \}
$$
where $\infty \in V^n/\left ( V^n - \theta(\phi)\right )$ is the
base point. The $G$-spaces $M_n^{-\tau}(e)$ form a symmetric
spectrum  over pointed $G$-spaces.

Since we have a bundle $\tilde{M}_n^{-\tau}(e) \rightarrow L(V,V^n)^G $
 and the connectivity of the space of equivariant embeddings
  from $V$ to $V^n$ tends to infinity, forgetting the symmetric
   structure, $M^{-\tau}(e)$  is equivalent to $\Sigma_V^\infty M^{\nu(e)}$.
   Where $\Sigma_V^\infty$ is the left adjoint to the functor that
    sends a spectrum over pointed $G$-spaces $\{X_n\}$ to the $G$-space $X_1$.

Similarly as in \citep{Cohen-Atiyahduality}, given an equivariant
 embedding $e: M \hookrightarrow V$ we can define symmetric ring spectra over pointed $G$-spaces $F(e)$ by,
$$
F_n(e) = \left \{ (\phi, \sigma) \mid \phi \in  L(V,V^n)^G, \sigma : \nu(e) \rightarrow S_x^{\oplus_n V} \right  \} / \{(\phi, \sigma_\infty)\}
$$
where, for $x \in \nu(e)$,  $S_x^{\oplus_n V}$ is the
compactification of a ball of radius $\epsilon$ around $\phi(e(x))$ in $V^n$
 and  $\sigma_\infty$ is the constant function at infinity.

The projection maps
$F_n(e) \cong L(V,V^n)^G_+ \wedge F(\nu(e),S^{\oplus_n V}) \rightarrow F(\nu(e),S^{\oplus_n V})$,
 together with the homotopy equivalence $F(\nu(e), S^{\oplus_n V}) \rightarrow F(M, S^{\oplus_n V})
$ induce a $\pi_*$-equivalence of symmetric ring spectra over pointed $G$-spaces
$$
\overline{\rho} : F(e) \rightarrow F(M,Sym(S^V)).
$$
Therefore, $F(e)$ is our model for the Spanier-Whitehead dual
of $M_+$ in the category $Sp^\Sigma(\mathcal{C},Sym(S^V))$.

Define the  Atiyah duality map
$$
\begin{array}{cc}
& \alpha : M^{-\tau}(e) \rightarrow F(e) \\
& (\phi, x) \rightarrow (\phi, \sigma_x)
\end{array}
$$
by $\sigma_x(y)=x-\phi(y)$ if $x$ belongs to the ball of radius
$\epsilon$ around $\phi(y)$, and to $\infty$ otherwise. This is a map of symmetric ring spectra.

The map $\alpha$ is compatible with the classical equivariant Atiyah
duality map and therefore we have a commutative diagram of spectra over pointed $G$-spaces
$$
\xymatrix{
M^{-\tau}(e)  \ar[rr]^\alpha &  & F(e) \ar[d]^{\overline{\rho}} \\
\Sigma_V^\infty M^{\nu(e)} \ar[u] \ar@{.>}[rr]& & F(M,Sym(S^V))
}
$$
where the vertical arrows are $\pi_*$-equivalences of symmetric spectra over pointed $G$-spaces and the dotted arrow is the equivariant Atiyah duality equivalence \cite{mayalask}, proving that $\alpha$ is a
$\pi_*$-equivalence of symmetric ring spectra over pointed $G$-spaces.

By an equivariant version of the Whitney embedding theorem \cite{equivdifftop}, the homotopy type of the spectra
$\Sigma_V^\infty M^{\nu(e)}$
does not depend on the embedding and we will denote it by $M^{-TM}$ when no confusion arises.

The prolongation of the singular simplicial set functor gives for
every symmetric ring spectrum over pointed $G$-spaces, a
(simplicial)  $H\mathbb{Z}$-symmetric algebra spectra over pointed
$G$-simplicial sets. In \citep{Shipley04hz-algebraspectra}, the
comparison between $H\mathbb{Z}$-symmetric modules and chain
complexes is established. There is a zig-zag of weak monoidal
Quillen equivalences  between (simplicial) $H\mathbb{Z}$-symmetric
algebra spectra and differential graded algebras, and similarly,
between their module categories. In our setting, for every
symmetric spectrum over pointed $G$-spaces we have a chain complex
with a $G$-action, and for every symmetric ring spectra (module)
over pointed $G$-spaces, we have a differential graded algebra
(module) with $G$-action.

Recall that the multiplication map
$$
\mu : M^{-\tau}(e) \wedge M^{-\tau}(e) \rightarrow M^{-\tau}(e)
$$
is dual to the inclusion map
$$
\begin{array}{cc}
\Delta: & M \rightarrow M \times M \\
& x \rightarrow (x,x)
\end{array}
$$
that at the level of chains
$$\Delta^*: C_{-*}(M^{-\tau}(e)) \otimes C_{-*}(M^{-\tau}(e)) \to C_{-*}(M^{-\tau}(e))$$
is dual to the cup product on cochains
$$\Delta^* : C^*(M) \otimes C^*(M) \to C^*(M).$$

Now, for $g \in G$, considering the twisted diagonals
$$
\begin{array}{cc}
\Delta_g: & M \rightarrow M \times M \\
& x \rightarrow (x,xg)
\end{array}$$
we have dual maps
$$
l_g : M^{-\tau}(e) \wedge M^{-\tau}(e) \rightarrow M^{-\tau}(e)
$$
and
$$
l_g : F(e) \wedge F(e) \rightarrow F(e)
$$
After applying chains (under the Atiyah equivalence) these maps correspond to the map on cochains,
$$
\begin{array}{cc}
\Delta_g^* & : C^*(M) \otimes C^*(M) \rightarrow C^*(M) \\
& a \otimes b \rightarrow a\cup g(b)
\end{array}
$$
The map $l_g$ gives a bimodule structure on $M^{-\tau}(e)$ over the
symmetric ring spectrum $M^{-\tau}(e)$ that we will denote by $M^{-\tau}(e)_g$
and a bimodule structure on $F(e)$ over $F(e)$ that we will denote $F(e)_g$.

Note that for the twisted diagonals to be equivariant we have to restrict
 to $C_G(g)$, the centralizer of $g$ in $G$.

The $G$-module differential graded bimodule $C_{-*}(M^{-\tau}(e))$
over $C_{-*}(M^{-\tau}(e))$ induced by $l_g$ and the usual multiplication,
 will be denoted by $C_{-*}(M^{-\tau}(e))_g$. Similarly, using $l_g$ we can
  twist the bimodule structures on $C_{-*}(M^{-\tau}(e))$ over $C_{-*}(F(e))$ and
  using $\Delta_g$ we can twist the bimodule structures on $C^*(M)$ over $C^*(M)$.
   We will denote these bimodule structures $C_{-*}(M^{-\tau}(e))_g$ and $C^*(M)_g$ respectively.

Consider the maps on Thom spectra
$$
\Delta^* :  M^{-\tau}(e) \rightarrow M^{-\tau}(e) \wedge \nu(e)_+ \ \ \ \Delta_g^* :  M^{-\tau}(e) \rightarrow \nu(e)_+ \wedge M^{-\tau}(e)
$$
induced by the map of bundles $ \Delta_* : \nu(e) \rightarrow \pi_1^*(\nu(e))$
and $\Delta_{g*} : \nu(e) \rightarrow    \pi_2^*(\nu(e))$ over the diagonal maps
$\Delta: M \to M\times M$ and $\Delta_g: M \to M\times M$. These induce a
bimodule structure on $C_{-*}(M^{-\tau}(e))$ over $C^*(\nu(e))$. We will denote
 this bimodule $C_{-*}(M^{-\tau}(e))_g$,

We have a $C_G(g)$-equivariant chain homotopy induced by the evaluation map
$$
ev_* : C_{-*}(F(\nu(e),Sym(S^V))) \rightarrow C^*(\nu(e))
$$
that gives an equivariant chain equivalences of differential graded algebras
$$
C_{-*}(F(e)) \rightarrow C_{-*} (F(\nu(e),Sym(S^V)) \rightarrow C^*(\nu(e))
$$
compatible with the bimodule structures on $C_{-*}(M^{-\tau}(e))_g$.

Since the map $\alpha$ is a $\pi_*$-equivalence of symmetric ring spectra
 we have equivariant chain homotopy ring homomorphisms
$$
 C_{-*}(M^{-\tau}(e)) \rightarrow C^*(F(M,Sym(S^V))) \rightarrow C^*(M)
$$
which induces the usual Poincare duality map.

\begin{theorem} \label{mio} Given a manifold $M$ with an action of a finite
group $G$ embedded equivariantly in a representation $V$, then there are
 $C_G(g)$-equivariant chain homotopy equivalences of differential graded algebras
$$
C_{-*}(F(e))  \rightarrow C^*(\nu(e))
$$
and
$$
C_{-*}(M^{-\tau}(e)) \rightarrow C_{-*}(F(e))
$$
which are compatible with the bimodule structures on $C_{-*}(M^{-\tau}(e))_g$.
\end{theorem}

\subsection{Orbifold string topology} \label{subsection OST} This section
 is based on the construction done in \citep{LUX} of the orbifold loop spectra $P_GM^{-TM}$

Now, for $t \in \real$ let $e_t :P_gM \to M$ be the evaluation of a map at the time $t$: $e_t(f) := f(t)$.
Consider the space of composable maps
$$P_gM \fiberprod{1}{0} P_hM =\{ (\phi , \psi) \in P_gM \times P_hM  | \phi(1)=\psi(0) \}$$
 and note that they fit into the  pullback square
$$\xymatrix{ P_gM \fiberprod{1}{0} P_hM \ar[d]^{e_0 \circ \pi_1} \ar[r] & P_gM \times P_hM \ar[d]^{e_0 \times e_0}\\
 M \ar[r]^{\Delta_g} & M \times M
}$$
where the downward arrow $e_0 \circ \pi_1$ is the evaluation at zero of the first map.

The normal bundle of the upper horizontal map becomes isomorphic to the pullback
under $e_0\circ \pi_1$ of the normal bundle of the map $\Delta_g$ and therefore
we can construct the Thom-Pontrjagin collapse maps that makes the diagram commutative

$$\xymatrix{ P_gM \times P_hM \ar[d]^{e_0 \times e_0} \ar[r] & P_gM \fiberprod{1}{0} P_hM^{\pi_1^*e_0^*N_g} \ar[d]^{e_0 \circ \pi_1}   \\
 M \times M \ar[r] & M^{N_g}
}$$
and by inverting the tangent bundle of $M\times M$ on the lower left hand we obtain the commutative square
$$\xymatrix{ P_gM^{-e_0^*TM} \wedge P_hM^{-e_0^*TM} \ar[d]^{e_0 \times e_0} \ar[r] & P_gM \fiberprod{1}{0} P_hM^{-\pi_1^*e_0^*TM} \ar[d]^{e_0 \circ \pi_1}   \\
 {M^{-TM} \wedge M^{-TM}} \ar[r]^{\Delta_g^*} & M^{-TM}.
}$$

The concatenation of paths in $P_gM \fiberprod{1}{0} P_hM$ produces a map
$$P_gM \fiberprod{1}{0} P_hM \to P_{gh}M$$
which induces the map of spectra
$$P_gM \fiberprod{1}{0} P_hM^{-\pi_1^*e_0^*TM}\to P_{gh}M^{-e_0^*TM}$$
that composed with the upper horizontal map of diagram defines the map
$$\mu_{g,h}: P_gM^{-e_0^*TM} \wedge P_hM^{-e_0^*TM} \to P_{gh}M^{-e_0^*TM}$$
which fits into the commutative diagram of spectra
 $$\xymatrix{ P_gM^{-e_0^*TM} \wedge P_hM^{-e_0^*TM} \ar[d]^{e_0 \times e_0} \ar[rr]^<<<<<<<<{\mu_{g,h}} && P_{gh}M^{-e_0^*TM} \ar[d]^{e_0 }   \\
 {M^{-TM} \wedge M^{-TM}} \ar[rr]^<<<<<<<<<<<<<<{\Delta_g^*} && M^{-TM}.
}$$

If we denote the spectrum $P_gM^{-TM}$ as
$$P_gM^{-TM} := P_gM^{-e_0^*TM},$$
we define
$$P_GM^{-TM} : =\bigsqcup_{g \in G} P_gM^{-TM}$$
which by assembling the maps $\mu_{g,h}$ we define a map
\begin{eqnarray} \label{map mu}
\mu : P_GM^{-TM} \wedge P_GM^{-TM} \to P_GM^{-TM}.
\end{eqnarray}
where we have denoted
$$P_GM^{-TM} \wedge P_GM^{-TM} := \bigsqcup_{(g,h) \in G \times G} P_gM^{-TM} \wedge P_hM^{-TM}.$$

Note that the compositions $\mu \circ (\mu \times 1) $ and $ \mu \circ ( 1 \times \mu)$ do not agree,
 but they are homotopically equivalent.

Recall from Definition \ref{definition loop groupoid} that $P_GM$ is a $G$-space where
for $k \in G$ and $f \in P_gM$ we have that
 $fk \in P_{k^{-1}gk}M$. Because the embedding of $M$ to the representation
 $e:M \hookrightarrow V$ is $G$-equivariant, the spectra $M^{-TM}$
acquires an action of $G$, and moreover, for $k \in G$ there is an induced
map $P_gM^{-TM} \to P_{k^{-1}gk}M^{-TM}$ of spectra, where
the based point in $P_gM^{-TM}$ is mapped to the based point in $P_{k^{-1}gk}M^{-TM}$.

It follows that for $k \in G$ we have the commutative diagram
$$\xymatrix{ P_gM^{-TM} \wedge P_hM^{-TM} \ar[d]^{k \times k} \ar[rrr]^<<<<<<<<<<<<<<<<<<<{\mu_{g,h}} &&& P_{gh}M^{-TM} \ar[d]^{k }   \\
 P_{k^{-1}gk}M^{-e_0^*TM} \wedge P_{k^{-1}hk}M^{-TM}  \ar[rrr]^<<<<<<<<<<<<<<<{\mu_{k^{-1}gk,k^{-1}hk}} &&& P_{k^{-1}ghk}M^{-TM}
 }$$
 which implies that the map $\mu: P_GM^{TM} \wedge P_GM^{-TM} \to P_GM^{-TM}$ is $G$ equivariant.

 Therefore we could think of $P_GM^{-TM}$ as ring spectrum with a $G$-action in the
  homotopy category, taking homology, we get a  $G$-module ring $H_*(P_GM^{-TM};\integer)$
   whose product structure was called the $G$-string product in \citep{LUX},
   and the induced ring structure on the invariant set
 $$H_*(P_GM^{-TM}; \rational)^G$$
 was called in \citep{LUX} the {\bf orbifold string topology ring}.

\subsection{Cosimplicial spectra} \label{section cosimplicial spectra}

In this section we will describe a cosimplicial model for the spectra $P_gM^{-TM}$
that will allow us to give a natural way
of relating the singular chains  $C_*(P_gM^{-TM})$ to the complex $\HHom_{{C^*}^e}(B({C^*}),{C^*_g})$.
This section is a generalization of section 3 of \citep{CohenJones} and we will mimic their construction.

Recall from (\ref{functions phi_k}) the functions $\phi_k$
\begin{eqnarray*}
\phi_k \colon \Delta_k \times P_gX & \to & M^{k+1}=\Pp_gM({\gr n}) \\
(t_1, \dots, t_k) \times \gamma & \mapsto & (\gamma(0),
\gamma(t_1), \dots, \gamma(t_k)). \nonumber
\end{eqnarray*}
and consider the commutative diagram
$$\xymatrix{ \Delta_k \times P_gM \ar[d]^{e}\ar[r]^>>>>>{\phi_k} & M^{k+1}\ar[d]^{\pi_1}\\
M \ar[r]^= &M
}$$
where $e((t_0, t_1, \dots, t_k), \gamma) \mapsto \gamma(0)$ and the right hand vertical
 map is the projection on the first coordinate.
Pulling back the virtual bundle $-TM$ under $e$ and $\pi_1$ we get a map of Thom spectra,
 that by abuse of notation we still call $\phi_k$
\begin{eqnarray*}
\phi_k : (\Delta_k)_+ \wedge P_gM^{-TM} \To M^{-TM} \wedge (M^k)_+.
\end{eqnarray*}

Taking adjoints, we get a map of spectra
\begin{eqnarray*}
\phi : P_gM^{-TM} \To \prod_k \mbox{Map}((\Delta_k)_+, M^{-TM} \wedge (M^k)_+)
\end{eqnarray*}
that is just the induced map of Thom spectra of the map $$\phi: P_gM \to \prod_{k \geq 0}
\mbox{Map}(\Delta_k, M^{k+1})$$ described in Lemma \ref{lemma total map of spectra}.

We are now ready to define the cosimplicial spectrum $\Pf_gM$; it will be the cosimplicial Thom
 spectrum of the cosimplicial
virtual bundle $-TM$ on $\Pp_gM$. More explicitly, let $\Pf_gM$ be the cosimplicial spectrum
 whose $k$-simplices are the spectrum
$$\Pf_gM_k := M^{-TM} \wedge (M^k)_+$$
and whose coface and codegeneracy maps are:
\begin{align*}
\delta_0(u;x_1, \dots, x_{k-1}) &= (u;y,x_1, \dots, x_{k-1})\\
\delta_i(u;x_1, \dots, x_{k-1}) &=  (u;x_1, \dots x_{i-1}, x_i, x_i, x_{i+1}, \dots , x_{k-1}),  \ \ \ 1 \leq i \leq k-1\\
\delta_k(u;x_1, \dots, x_{k-1}) &= (v; x_1, \dots , x_{k-1}, z)\\
\sigma_i(u;x_1, \dots, x_{k+1}) &= (u; x_1, \dots, x_i, x_{i+2}, \dots, x_{k+1}) \ \ \ \ 0 \leq i \leq k
\end{align*}
with $\mu(u)=(u,y)$ and $\nu(u) = (z,v)$ where
$$\mu : M^{-TM} \to M^{-TM} \wedge M_+ \ \ \ \ \ \ \ \  \nu: M^{-TM} \to M_+ \wedge M^{-TM} $$
are the maps of Thom spectra induced over the diagonal maps $\Delta: M \to M\times M$ and
$\Delta_g: M \to M\times M$  by the maps of virtual bundles $\Delta_*: -TM \to -\pi_1^*TM$
 and $\Delta_{g*} : -TM \to -\pi_2^*TM$.

Let $|\Pf_gM|$ be the total spectrum of the cosimplicial spectrum $\Pf_gM$ i.e. it consists of sequences of maps
$\{\gamma_k\}$ in
$$\prod_k \mbox{Map}((\Delta_k)_+ ; M^{-TM} \wedge (M^k)_+)$$
which commute with the coface and codegeneracy maps. Therefore we can conclude that
by applying the Thom spectrum functor for the virtual bundle of Lemma \ref{lemma total map of spectra} we get
\begin{proposition} \label{proposition equivalence phi}
The map $$ \phi : P_gM^{-TM} \To \prod_k \mbox{Map}((\Delta_k)_+, M^{-TM} \wedge (M^k)_+)$$
induces a homeomorphism  between the spectrum $P_gM^{-TM}$ and the total spectrum of $\Pf_gM$
$$\phi : P_gM^{-TM} \stackrel{\cong}{\To} |\Pf_gM|.$$
\end{proposition}

As a consequence of Lemma \ref{lemma total-cochains = cochains-total} after passing to
Thom spectra, we get that the maps $\phi_k :(\Delta_k)_+ \wedge P_gM^{-TM} \to M^{-TM} \wedge (M^k)_+$
 define maps of cochains
$$C^*(M^{-TM} \wedge (M^k)_+)[k] \to C^*(P_gM^{-TM})$$
which assemble to give a map from the total complex of the simplicial cochain
 complex $|C^*(\Pf_gM)|$ to the cochain complex of the total spectrum $|\Pf_gM|$; therefore we have
\begin{lemma} \label{lemma total-cochains = cochains-total of spectrum}
There is a homomorphism of graded complexes
$$|C^*(\Pf_gM) | \To C^*(P_gM^{-TM})$$
that when $M$ is a connected and simply connected manifold it becomes a quasi-isomorphism.
\end{lemma}

If we denote
$$|C^*(\Pf_gM) |^\vee:= \HHom_\integer(|C^*(\Pf_gM) |, \integer)$$ then let us show that
\begin{lemma} \label{lemma total cochain= Hom(BA,A)}
There is a map of graded complexes
$$|C^*(\Pf_gM) |^\vee \to \HHom_{{C^*}^e}(B({C^*}),{C^*_g})$$
that moreover is a quasi-isomorphism.
\end{lemma}
\begin{proof}

The cochains of the $k$ simplices of $\Pf_gM$ become chain homotopy equivalent to
\begin{align*}
C^*(M^{-TM} \wedge (M^k)_+) \cong C^*(M)^k \otimes C^*(M^{-TM})
\end{align*}
and after dualizing we get that
\begin{align*}
C^*(M^{-TM} \wedge (M^k)_+)[k]^\vee \cong &\HHom_\integer(C^*(M)[1]^{\otimes k} \otimes C^*(M^{-TM}), \integer)\\
\cong & \HHom_\integer (C^*(M)[1]^{\otimes k}, C_{-*}(M^{-TM}))
\end{align*}

Now we want to use the fact that Atiyah duality  produces an equivalence of symmetric spectra,
that induces a chain homotopy equivalence
$$\alpha_* : C_{-*}(M^{-TM}) \to C^*(M)$$
compatible with the $C^*(M)^{e}$-module structure of $C_{-*}(M^{-\tau}(e))$, this is the
content of Theorem \ref{mio}. For this, let us modify (by a homotopy equivalent model) the
 cosimplicial spectrum $\Pf_gM$. Its $k$-cosimplices are
$$
M^{-\tau}(e) \wedge (\nu(e)_+)^k
$$
with the corresponding codegeneracy and coface maps defined using $\mu$ and $\nu_g$.
As before, the chains of the $k$-cosimplices are given by
\begin{align*}
C^*(M^{-\tau}(e) \wedge (\nu(e)^k)_+)[k]^\vee
\cong  &    Hom(C^*(\nu(e))[1]^{\otimes k},C_{-*}(M^{-\tau}(e))_g) \\
\cong & \HHom_{{C^*}^e}({C^*} \otimes {C^*}[1]^{\otimes k} \otimes {C^*}, C_{-*}(M^{-\tau}(e))_g) \\
\cong & CH^*(C^*(\nu(e)),C_{-*}(M^{-\tau}(e))_g)[k]
\end{align*}

Using the first part of Theorem \ref{mio} we can replace $C^*(\nu(e))$ by $C_{-*}(F(e))$
 to obtain a $C_G(g)$-equivariant chain equivalence
$$
|C^*(\Pf_gM) |^\vee \rightarrow CH^*(C_{-*}(F(e)),C_{-*}(M^{-\tau}(e))_g).
$$

Now, by the second  part of Theorem \ref{mio}, we have a $C_G(g)$-equivariant chain
homotopy equivalence
$$
CH^*(C_{-*}(F(e)),C_{-*}(M^{-\tau}(e))_g) \rightarrow  CH^*(C_{-*}(M^{-\tau}(e)),C_{-*}(M^{-\tau}(e))_g)
$$
which all together give a $C_G(g)$-equivariant chain homotopy equivalence
\begin{equation} \label{ec}
|C^*(\Pf_gM) |^\vee \rightarrow CH^*(C_{-*}(M^{-\tau}(e)),C_{-*}(M^{-\tau}(e))_g).
\end{equation}

Since we have a $\pi_*$-equivalence of symmetric ring spectra over pointed $G$-spaces
$\alpha : M^{-\tau}(e) \rightarrow F(e)$ that is compatible with the twisted bimodule
 structure on $M^{-\tau}(e)_g$ and $F(e)_g$ (both bimodule structures come as duals to
  the twisted diagonals), then it induces an equivalence of their topological Hochschild cohomologies
$$
THH^*(M^{-\tau}(e),M^{-\tau}(e)_g) \cong THH^*(F(e),F(e)_g).
$$
By applying the chain functor we obtain $C_G(g)$-equivariant chain homotopies
\begin{equation} \label{ec2}
CH^*(C_{-*}(M^{-\tau}(e)),C_{-*}(M^{-\tau}(e))_g) \cong CH^*(C^*(M),C^*(M)_g)
\end{equation}
Finally putting together the isomorphisms \ref{ec} and \ref{ec2} we obtain the desired
result, a $C_G(g)$-equivariant chain homotopy equivalence
$$
|C^*(\Pf_gM) |^\vee \rightarrow CH^*(C^*(M),C^*(M)_g) = \HHom_{{C^*}^e}(B({C^*}),{C^*_g})
$$
\end{proof}

Lemma \ref{lemma total-cochains = cochains-total of spectrum} together with Lemma \ref{lemma total cochain= Hom(BA,A)}
allow us to conclude that
\begin{theorem}
For $M$ a connected and simply connected compact manifold, there exists a homomorphism of graded complexes
$$ C_* (P_gM^{-TM} )\stackrel{\simeq}{\To} \HHom_{{C^*}^e}(B({C^*}),{C^*_g})$$
that moreover is a quasi-isomorphism, and therefore induces an
isomorphism of graded groups
$$ H_* (P_gM^{-TM} )\stackrel{\cong}{\To} H^* \HHom_{{C^*}^e}(B({C^*}),{C^*_g}) .$$
\end{theorem}

From the previous theorem, assembling the maps for all $g$ in $G$, we can deduce that
\begin{cor} \label{cor q.i. C(P_GM) to Hom(BA,AG)}
For $M$ a connected and simply connected compact manifold, there exists a homomorphism of graded complexes
$$ C_* (P_GM^{-TM} )\stackrel{\simeq}{\To} \HHom_{{C^*}^e}(B({C^*}),{C^*}\#G)$$
that moreover is a quasi-isomorphism and $G$-equivariant, and therefore induces a $G$-equivariant
isomorphism of graded groups
$$ H_* (P_GM^{-TM} )\stackrel{\cong}{\To} H^*  \HHom_{{C^*}^e}(B({C^*}),{C^*}\#G) .$$
\end{cor}
\begin{proof}
The only thing left to prove is that the maps $P_gM^{-TM} \to |\Pf_gM|$ induce an equivariant map
$$P_GM^{-TM} \to \bigsqcup_{g\in G} |\Pf_gM|,$$
but this follows from the fact that the evaluation maps $e_t: P_GM \to M$ are $G$ equivariant.
\end{proof}

\subsection{Multiplicative structures} \label{section multiplicative structures}

In this section we will show how the map of Corollary \ref{cor q.i. C(P_GM) to Hom(BA,AG)}
is compatible with the multiplicative structure of $C_*(P_GM^{-TM})$ coming from the map
 $\mu$ of (\ref{map mu}), and the natural ring structure on $\HHom_{{C^*}^e}(B({C^*}),{C^*}\#G)$
  that was explained in section \ref{subsubsection B(G)}. To achieve this, we will endow the
  cosimplicial spectra $\Pf_gM$ with multiplicative maps $$\tilde{\mu}_{g,h} : \Pf_gM \wedge \Pf_hM \to \Pf_{gh}M$$
that once realized will be compatible with the maps $\mu_{g,h}$ and that after passing to
 chains, will realize the ring structure of $\HHom_{{C^*}^e}(B({C^*}),{C^*}\#G)$.

Let us define the maps
\begin{align*}
\tilde{\mu}_{g,h}^{k,l} : \left( M^{-TM} \wedge (M^k)_+ \right) \wedge \left( M^{-TM} \wedge (M^l)_+ \right) & \To  M^{-TM} \wedge (M^{k+l})_+ \\
(u; x_1, \dots, x_k) \wedge (v; y_1, \dots, y_l)  \mapsto (\Delta_g^*(u,v)&; x_1, \dots, x_k,y_1, \dots, y_l)
\end{align*}
that define maps of the simplices
$$\tilde{\mu}_{g,h}^{k,l}  : (\Pf_gM)_k \wedge (\Pf_hM)_l \to (\Pf_{gh}M)_{k+l}$$
which commute with the coface and codegeneray operators, and therefore induce maps of cosimplicial spectra
$$\tilde{\mu}_{g,h} : \Pf_gM \wedge \Pf_hM \to \Pf_{gh}M.$$
Taking the total spectrum we get maps of spectra
$$|\tilde{\mu}_{g,h}| : |\Pf_gM| \wedge |\Pf_hM| \to |\Pf_{gh}M|$$
that induce pairings that are  $A_\infty$-associative.

Applying the chains functor to the map $\tilde{\mu}_{g,h}^{k,l}$ we get the map
\begin{align*}
(\tilde{\mu}_{g,h}^{k,l})_* : (C_{-*}(M)^k \otimes C_{-*}(M^{-TM})) \otimes (C_{-*}(M)^l \otimes C_{-*}(M^{-TM})) &  \to\\ C_{-*}(M)^{k+l}& \otimes C_{-*}(M^{-TM})
\end{align*}
that by Theorem \ref{mio} induces the map
\begin{align*}
(\tilde{\mu}_{g,h}^{k,l})_* : (C_{-*}(M)^k \otimes C^*(M)) \otimes (C_{-*}(M & )^l \otimes C^*(M))  \to C_{-*}(M)^{k+l} \otimes C^*(M)\\
(a_1 \otimes \cdots \otimes a_k \otimes \alpha) \otimes (b_1 \otimes \cdots b_l \otimes \beta)&  \mapsto \\ a_1 \otimes & \cdots \otimes  a_k \otimes  b_1 \otimes \cdots \otimes b_l \otimes \Delta_g^*(\alpha \otimes \beta) \\
 = a_1 \otimes & \cdots \otimes  a_k \otimes  b_1 \otimes \cdots \otimes b_l \otimes \alpha \cdot g(\beta).
\end{align*}

Composing with the natural isomorphisms
$$C_{-*}(M)^k \otimes C^*(M) \cong \HHom_{C^*(M)^e}(C^*(M)^{k+2}, C^*(M))$$
we see that the map $(\tilde{\mu}_{g,h}^{k,l})_*$ induces the multiplicative structure that for
$$\phi_g \in \HHom_{C^*(M)^e}(C^*(M)^{k+2}, C^*(M)), \ \ \psi_h \in \HHom_{C^*(M)^e}(C^*(M)^{l+2}, C^*(M))$$
defines the function$$(a_0| \dots| a_{k+l+1}) \mapsto (-1)^{|\psi_h|\varepsilon_k} \phi_g(a_0| \dots |a_k|1) g(\psi_h(1|a_{k+1}| \dots |a_{k+l+1})$$
which agrees with the definition of the product structure of $\HHom_{{C^*}^e}(B({C^*}),{C^*} \# G)$ done in section \ref{subsubsection B(G)}.

We have then
\begin{lemma} \label{lemma multiplicative 1}
The maps of Lemma \ref{lemma total cochain= Hom(BA,A)}
$$|C^*(\Pf_gM) |^\vee \to \HHom_{{C^*}^e}(B({C^*}),{C^*_g})$$
 are compatible with the multiplicative structures on $\bigoplus_{g \in G} |C^*(\Pf_gM) |^\vee$ and  $\HHom_{{C^*}^e}(B({C^*}),{C^*} \#G)$, which on the first term, the multiplicative structure is  induced by the maps $(\tilde{\mu}_{g,h}^{k,l})_*$. In particular we have
 an isomorphism of rings
 $$\bigoplus_{g \in G} H^*(|C^*(\Pf_gM) |^\vee) \stackrel{\cong}{\To}  H^* \HHom_{{C^*}^e}(B({C^*}),{C^*}\#G).$$
 \end{lemma}

We are left with providing the relationship between the maps $\mu_{g,h}$ and the maps
$\tilde{\mu}_{g,h}$; this will be achieved with the following theorem
\begin{theorem} \label{theorem multiplicative 2}
The multiplicative structures that the maps $\mu_{g,h}$ and $\tilde{\mu}_{g,h}$ define
 are compatible in the sense that the following diagram homotopy commutes:
$$\xymatrix{
P_gM^{-TM} \wedge P_hM^{-TM}  \ar[rr]^>>>>>>>>{\mu_{g,h}} \ar[d]^{\phi \wedge \phi} && P_{gh}M^{-TM} \ar[d]^\phi\\
|\Pf_gM| \wedge |\Pf_hM| \ar[rr]^>>>>>>>>>{\tilde{\mu}_{g,h}} && |\Pf_{gh}M|
}$$where $\phi$ is the map of proposition \ref{proposition equivalence phi}.
\end{theorem}
\begin{proof}
The proof is almost identical to the proof of Theorem 13 on \citep{CohenJones}.
The only difference is that whenever it is used the diagonal map $\Delta$ on the proof
of Theorem 13 of \citep{CohenJones} it needs to be replaced by the twisted diagonal
$\Delta_g$. We will not reproduce the proof here.
\end{proof}

Lemma \ref{lemma multiplicative 1} together with Theorem \ref{theorem multiplicative 2}
 imply that the $G$-equivariant quasi-isomorphism of Corollary \ref{cor q.i. C(P_GM) to Hom(BA,AG)}
  is compatible with the multiplicative structures previously defined, and therefore we can finish this section with
\begin{theorem}
\label{theorem q.i. A_infty C(P_GM) to Hom(BA,AG)}
For $M$ a connected and simply connected compact manifold, there exists a quasi-isomorphism of graded complexes
$$\Phi: C_* (P_GM^{-TM} )\stackrel{\simeq}{\To} \HHom_{{C^*}^e}(B({C^*}),{C^*}\#G)$$
which is $G$-equivariant, that furthermore makes the following diagram commute
$$\xymatrix{C_* (P_GM^{-TM} ) \otimes C_* (P_GM^{-TM} ) \ar[r]^>>>>>>>>>{\mu_*} \ar[d]^{\phi_* \otimes \phi_*} & C_* (P_GM^{-TM} ) \ar[d]^{\phi_*} \\
\HHom_{{C^*}^e}(B({C^*}),{C^*} \# G) \otimes \HHom_{{C^*}^e}(B({C^*}),{C^*} \# G) \ar[r]^>>>>\cdot & \HHom_{{C^*}^e}(B({C^*}),{C^*} \# G).
}$$
Hence, $\Phi$ induces a $G$-equivariant
isomorphism of rings
$$ \Phi : H_* (P_GM^{-TM} )\stackrel{\cong}{\To}  H^* \HHom_{{C^*}^e}(B({C^*}),{C^*}\#G) .$$
\end{theorem}

The previous Theorem \ref{theorem q.i. A_infty C(P_GM) to Hom(BA,AG)}, together
with Theorem \ref{theorem decomposition HH} and the results of the section \ref{Appendix B} imply that
\begin{theorem} \label{theorem HH=Ext(Z,PGM)}
For $M$ a connected and simply connected and $G$ a finite group acting on $M$, then there
exits an isomorphism of graded rings
$$HH^*(C^*(M) \#G, C^*(M) \#G) \cong \Ext^*_{\integer G}( \integer, C_*(P_GM^{-TM})).$$
\end{theorem}

\section{Homotopical realization for $HH^*(C^*(M) \#G, C^*(M) \#G)$} \label{section homotopical}

From Theorem \ref{theorem HH=Ext(Z,PGM)} we know that have an isomorphism of rings
$$HH^*(C^*(M) \#G, C^*(M) \#G) \cong \Ext^*_{\integer G}( \integer, C_*(P_GM^{-TM}))$$
which can be further extended to an isomorphism of rings
\begin{eqnarray} \label{isomorphism C^*EG}
HH^*(C^*(M) \#G, C^*(M) \#G) \cong H^*\left( C^*(EG) \otimes C_*(P_GM^{-TM})\right)^G\end{eqnarray}
as we know that
\begin{align}
\Ext^*_{\integer G}( \integer, C_*(P_GM^{-TM})) & \cong H^* \HHom_{\integer G}( C_*(EG), C_*(P_GM^{-TM}) ) \nonumber\\
& \cong H^* \HHom_{\integer G}(  \integer, C^*(EG) \otimes C_*(P_GM^{-TM}) ) \nonumber \\
 & = H^*\left( C^*(EG) \otimes C_*(P_GM^{-TM})\right)^G. \label{iso ext(Z), C^(EG)}
\end{align}

We would like to have a topological construction associated to the Hochschild cohomology
 of $C^*(M) \#G$. One deterrent for its existence comes
from the fact that on the right hand side of (\ref{isomorphism C^*EG}) we have a mixture
of chain with cochain complexes.

To overcome this problem we will construct a pro-ring spectrum $EG^{-TEG}$ associated to
$EG$, that will allow us to change the cochain complex $C^*(EG)$ by the chain complex of $EG^{-TEG}$.
With this idea in mind, let us generalize the construction of the string topology of $BG$
that can be found in \citep{GruherSalvatore, GruherWesterland} to the orbifold case.

Let $EG_1 \subset \cdots \subset EG_n \subset EG_{n+1} \subset \cdots EG$ be a finite
 dimensional manifold approximation of the universal $G$-principal bundle $EG \to BG$. Consider the maps
$$P_GM \times_G EG_n \stackrel{e_0}{\to} M\times_G EG_n$$
where by abuse of notation we denote the evaluation at the time $t$ of a pair $(f, \lambda) \in P_GM \times_G EG_n$ also by $e_t$. For convinience let us call $M_n = M\times_G EG_n$.

Take the Thom spectra
$$\left( P_GM \times_G EG_n \right)^{-e_0^*TM_n}$$
and let us show that indeed it is a ring spectra. Consider the diagram
$$ \xymatrix{ (P_GM \fiberprod{1}{0} P_GM) \times_G EG_n  \ar[rr]^{\widetilde{\Delta}} \ar[d]^{e_0} &&(P_GM \times_G EG_n) \times (P_GM \times_G EG_n) \ar[d]^{e_0 \times e_0}  \\
  M\times_G EG_n \ar[rr]^{\Delta}&&
(M\times_G EG_n) \times (M \times_G EG_n)
}$$
where $\Delta$ is the diagonal inclusion and $P_GM \fiberprod{1}{0} P_GM$ fits in the pullback square
$$\xymatrix{
P_GM \fiberprod{1}{0} P_GM \ar[d]^{e_\infty} \ar[r] & P_GM \times P_GM \ar[d]^{e_1 \times e_0}\\
M \ar[r]^{\rm diag} & M \times M.
}$$ As the normal bundle of the inclusion $P_GM \fiberprod{1}{0} P_GM \to P_GM \times P_GM$ is isomorphic
to $e_\infty^*TM$, then the normal bundle of the inclusion
$$ \widetilde{\Delta}:  (P_GM \fiberprod{1}{0} P_GM) \times_G EG_n \to  (P_GM \times_G EG_n) \times (P_GM \times_G EG_n)$$
is isomorphic to $e_\infty^* TM_n$.
Then by the Thom-Pontryagin construction we have a map
$$(P_GM \times_G EG_n) \times (P_GM \times_G EG_n) \to  \left( (P_GM \fiberprod{1}{0} P_GM) \times_G EG_n \right)^{e_\infty^*TM_n}$$
which induces a map of spectra
$$ \scriptstyle \nu: (P_GM \times_G EG_n)^{-e_1^*TM_n} \wedge (P_GM \times_G EG_n)^{-e_0^*TM_n}
\longrightarrow \left( ( P_GM \fiberprod{1}{0} P_GM) \times_G EG_n \right)^{-e_\infty^*TM_n}.$$

Let us recall the concatenation map
\begin{eqnarray*}
\mu : P_GM \fiberprod{0}{1} P_GM & \to & P_GM \\
((\phi,g), (\psi, h)) & \mapsto & (\phi \circ \psi, gh)
\end{eqnarray*}
where
$$\phi \circ \psi (t):= \left\{ \begin{array}{lcl}
\phi(2t) & {\rm for} & 0 \leq t<\frac{1}{2}\\
\psi(2t-1) & {\rm for} & \frac{1}{2} \leq t \leq 1;
\end{array}  \right.$$

We have the commutative square
$$\xymatrix{
(P_GM \fiberprod{0}{1} P_GM) \times_G EG_n \ar[r]^>>>>\mu \ar[d]^{e_\infty} & P_GM \times_G EG_n \ar[d]^{e_{\frac{1}{2}}} \\
M\times_G EG_n \ar[r]^= & M\times_G EG_n
}$$
that induces a map on spectra
$$((P_GM \fiberprod{0}{1} P_GM) \times_G EG_n)^{-e_\infty^*T(M\times_G EG_n)} \stackrel{\overline{\mu}}{\longrightarrow} (P_GM \times_G EG_n)^{-e_{\frac{1}{2}}^*T(M\times_G EG_n)}$$
which composed with the map $\nu$ gives us a map of spectra
\begin{align*}
   (P_GM \times_G EG_n)^{-e_1^*T(M\times_G EG_n)}  \wedge (P_GM & \times_G EG_n )^{-e_0^*T(M\times_G EG_n)}\\
 & \longrightarrow  (P_GM \times_G EG_n)^{-e_{\frac{1}{2}}^*T(M\times_G EG_n)}
\end{align*}
Because all the maps $e_0, e_{\frac{1}{2}}, e_1$ are homotopy equivalent, then the bundles
$e_0^*TM_n \cong e_{\frac{1}{2}}^*TM_n \cong e_1^*TM_n$ are all isomorphic. Therefore we can construct a map of spectra
\begin{align*}
   (P_GM \times_G EG_n)^{-e_0^*T(M\times_G EG_n)}  \wedge (P_GM & \times_G EG_n )^{-e_0^*T(M\times_G EG_n)}\\
 & \longrightarrow  (P_GM \times_G EG_n)^{-e_0^*T(M\times_G EG_n)}
\end{align*}
that makes $(P_GM \times_G EG_n)^{-e_0^*TM_n}$ into a ring spectra.

The inclusions
$$\xymatrix{
\cdots \ar[r] & P_GM \times_G EG_n  \ar[r] \ar[d]  & P_GM \times_G EG_{n+1}\ar[d] \ar[r] & \cdots\\
\cdots \ar[r] &M\times_G EG_n \ar[r] &  M \times_G EG_{n+1} \ar[r] & \cdots
}$$
induce maps of ring spectra
\begin{eqnarray} \label{system of spectra} \\
(P_GM \times_G EG_n)^{-e_0^*T(M \times_G EG_n)} \stackrel{\rho^{n+1}_n}{\leftarrow} (P_GM \times_G EG_{n+1})^{-e_0^*T(M\times_G EG_{n+1})}. \nonumber
\end{eqnarray}
\begin{definition}
The previous system of ring spectra will be denoted by
$$\LL(M\times_G EG)^{-T(M\times_G EG)},$$
and we will call it {\bf the free loop space pro-ring spectrum associated to the orbifold}.
\end{definition}

Note that the maps of ring spectra in (\ref{system of spectra}) do not assemble into an inverse system of ring spectra. The reason for this is that each of the maps $\rho^{n+1}_n$ of ring spectra in (\ref{system of spectra}) depends explicitly on an equivariant embedding $EG_{n+1} \to V_{n+1}$ of $EG_{n+1}$ into a representation of $G$. Therefore the composition $\rho^n_{n-1} \circ \rho^{n+1}_n$ does not necessarily agree with the map of spectra
\begin{eqnarray*}
(P_GM \times_G EG_{n-1})^{-e_0^*T(M \times_G EG_{n-1})} \stackrel{\rho^{n+1}_{n-1}}{\leftarrow} (P_GM \times_G EG_{n+1})^{-e_0^*T(M \times_G EG_{n+1})}
\end{eqnarray*}
induced by the inclusion $EG_{n-1} \subset EG_{n+1}$.

Nevertheless, the maps $\rho_{n-1}^{n+1}$ and $\rho^n_{n-1} \circ \rho^{n+1}_n$ are
 homotopically equivalent, and therefore after applying the homology functor we indeed
  get an inverse system of graded rings. Therefore we define
\begin{definition}
The homology of the free loop space pro-ring spectrum associated to the orbifold $[M/G]$ is
$$H^{pro}_*(\LL(M\times_G EG)^{-T(M\times_G EG)}) := \lim_{{\leftarrow}{n}} H_*((P_GM \times_G EG_n)^{-e_0^*T(M \times_G EG_n)})$$
\end{definition}

We claim

\begin{theorem} \label{main theorem}
For $M$ connected and simply connected, and $G$ a finite group acting on
 $M$ then there is an isomorphism of graded rings
$$HH^*( C^*(M) \#G,C^*(M) \#G) \cong H^{\rm pro}_*\left(\LL(M\times_G EG)^{-T(M\times_G EG)}\right) $$
between the Hochschild cohomology of $C^*(M) \#G$ and the homology of the
 free loop space pro-ring spectrum associated to the orbifold $[M/G]$.
\end{theorem}
\begin{proof} Consider the following sequence of graded ring isomorphisms

\begin{align}
H^{pro}_*(\LL(M\times_G   EG &)^{-T(M  \times_G EG)})  \nonumber \\
:= & \lim_{{\leftarrow}{n}} H_*((P_GM \times_G EG_n)^{-e_0^*T(M \times_G EG_n)}) \label{line 1}\\
 \cong & \lim_{{\leftarrow}{n}} H_*\left(C_*\left((P_GM \times EG_n)^{-e_0^*T(M \times EG_n)}\right)^G\right)\label{line 2}\\
 \cong & \lim_{{\leftarrow}{n}} H_*\left(C_*\left(P_GM^{-TM} \wedge EG_n^{-TEG_n}\right)^G\right) \label{line 3}\\
  \cong & \lim_{{\leftarrow}{n}} H_*\left( \left(C_*(EG_n^{-TEG_n}) \otimes C_*(P_GM^{-TM}) \right)^G\right) \label{line 4}\\
  \cong & \lim_{{\leftarrow}{n}} H^*\left( \left(C^*(EG_n) \otimes C_*(P_GM^{-TM}) \right)^G\right)\label{line 5}\\
  \cong & H^*\left(\lim_{{\leftarrow}{n}}  \left(C^*(EG_n) \otimes C_*(P_GM^{-TM}) \right)^G\right) \label{line 6}\\
  \cong & H^*\left(  \left(C^*EG \otimes C_*(P_GM^{-TM}) \right)^G\right) \label{line 7}\\
  \cong & \Ext^*_{\integer G}\left( \integer, C_*(P_GM^{-TM}) \right)\label{line 8}\\
  \cong & HH^*(C^*(M) \#G, C^*(M)\#G) \nonumber
\end{align}
where the isomorphism between (\ref{line 1}) and (\ref{line 2}) is due to the fact that $G$ acts freely on $EG_n$;
the isomorphism between (\ref{line 2}) and (\ref{line 3}) is due to the fact that as ring spectra $X^{-TX} \wedge Y^{-TY}$ and $(X \times Y)^{-T(X \times Y)}$ are homotopic; the isomorphism between (\ref{line 3}) and (\ref{line 4}) follows by the Eilenberg-Zilber theorem; the isomorphism between (\ref{line 4}) and (\ref{line 5}) follows by S-duality between $C_*(EG_n^{-TEG_n})$ and $C^*(EG_n)$; the isomorphisms between (\ref{line 5}) and (\ref{line 6}) follows from the fact that the inverse system of graded rings
$$H^*\left( \left(C^*(EG_n) \otimes C_*(P_GM^{-TM}) \right)^G \right)$$
satisfies the Mittag-Leffler condition;
the isomorphism between (\ref{line 6}) and (\ref{line 7}) follows from the
 fact that the inverse system $H^*(EG_n)$ of graded rings satisfies the Mittag-Leffler condition;
   the isomorphism between (\ref{line 7}) and (\ref{line 8}) follows from
    (\ref{iso ext(Z), C^(EG)}); and the last isomorphism is proved in Theorem \ref{theorem HH=Ext(Z,PGM)}.

Then, the theorem follows from the previous isomorphisms.
\end{proof}

With Theorem \ref{main theorem} at hand, we define
\begin{definition}
The string topology ring associated to a global quotient orbifold $[M/G]$ is the graded ring
$$H^{pro}_*(\LL(M\times_G   EG)^{-T(M  \times_G EG)}; \integer).$$
\end{definition}
Note that the previous definition is indeed a homotopy invariant of the orbifold,
 and hence, it is well defined in the Morita equivalence class
of the orbifold $[M/G]$.
The isomorphism with the Hochschild cohomology ring of $C^*(M)\#G$ depended
explicitly on the fact that $H^1(M)=0$, and hence it is needed that $M$ be simply connected.

\section{Applications} \label{Applications}

\subsection{Rational coefficients}
When we use rational coefficients, the $G$-invariants functor is exact, and therefore
 there is no need in deriving it. Hence the string topology ring of the orbifold $[M/G]$ becomes
  isomorphic as graded rings to
$$H^{pro}_*\left(\LL(M\times_G EG)^{-T(M \times_G EG)}; \rational \right)\cong H_*(P_GM^{-TM} ; \rational)^G$$
which is precisely what was defined in \citep{LUX} as the orbifold string topology ring.

So, we have that in the case that $M$ is connected and simply connected, the orbifold
string topology of $[M/G]$ with coefficients
in $\rational$ is isomorphic to the ring
$$H_*(P_GM^{-TM} ; \rational)^G \cong HH^*(C^*(M;\rational) \#G, C^*(M;\rational)\#G).$$

\subsection{$M= $ point} The string topology for $[*/G]$ turns out to be a ring that
can be reproduced with the pull-push formalism (see \citep{Witherspoon}).
Let us see first that as a graded $\integer$-module, we have an isomorphism
$$H^{\rm pro}_*(\LL BG^{-TBG}) \cong \bigoplus_{(g)} H^*(BC(g))$$ where $(g)$ runs over
the conjugacy classes of elements in $G$ and $C(g)$ denotes the centralizer of $g$ in $G$

The groupoid $[P_GM/G]$ becomes simply $[G/G]$ where $G$ acts by conjugation on $G$, thus we obtain
$$ (G \times_G EG_n)^{-TBG_n} = \bigsqcup_{(g)} (EG_n/C(g))^{-\pi^*TBG_n} \cong \bigsqcup_{(g)} (BC(g)_n)^{-TBC(g)_n}$$ with $\pi : BC(g)_n:= EG_n/C(g) \to EG_n/G=BG_n$; the second equality follows from the fact that $\pi$ is a cover map and therefore  $TBC(g)_n \cong \pi^*TBG_n$.

Hence
\begin{eqnarray*}
H^{\rm pro}_*(\LL BG^{-TBG}) & = & \bigoplus_{(g)} \lim_{\leftarrow n} H_*\left((BC(g)_n)^{-TBC(g)_n}\right)\\
& = & \bigoplus_{(g)} \lim_{\leftarrow n} H^*(BC(g)_n)\\
& = & \bigoplus_{(g)}  H^*(BC(g)).
\end{eqnarray*}

Now let us see what is the induced ring structure: we have the maps
$$\xymatrix{
EG_n/C(g) \times EG_n/C(h)  & EG_n/(C(g) \cap C(h)) \ar[l]_<<<<\Delta \ar[r]  & EG_n/C(gh) \\
}$$
and all the pullbacks of the bundle  $TBG_n$ are isomorphic the corresponding tangent bundles. Note that the map $\Delta$ is injective
and therefore we can perform the Thom-Pontrjagin construction giving us the map in homology
$$H_*(EG_n/C(g) \times EG_n/C(h)) \to H_{*-k_n}(EG_n/(C(g) \cap C(h)))$$
with $k_n = \mbox{dim} (EG_n)$, which is Poincar\'e dual to the pull-back map in cohomology
$$\Delta^*:H^*(EG_n/C(g) \times EG_n/C(h)) \to H^*(EG_n/(C(g) \cap C(h))).$$

The natural map in homology
$$H_*(EG_n/(C(g) \cap C(h))) \to H_*(EG_n/C(gh))$$
is Poincar\'e dual to the push-forward map in cohomology
$$H^*(EG_n/(C(g) \cap C(h))) \to H^*(EG_n/C(gh))$$
that defines the induction map
$$H^*(B(C(g) \cap C(h))) \to H^*(BC(gh)).$$

We therefore see that the ring structure in $H^{\rm pro}_*(\LL BG^{-TBG})$ is obtained by
 taking classes in $H^*(BC(g))$ and $H^*(BC(h))$ respectively,
pulling them back to $H^*B(C(g) \cap C(h))$ and then pushing them forward to $H^*(BC(gh))$.
 This procedure is what is known as the pull-push formalism, and it is a well known fact among
 algebraists that the ring structure $HH^*(\integer G, \integer G)$ could be recovered
  with this formalism (see Example 2.7 of \citep{Witherspoon} an the references therein).

Note that in the case that $G$ is abelian, $BG = K(G,1) = \Omega K(G,2)$ is a
topological group and therefore, $LBG \cong G \times BG$, and we have
$$HH^*(\integer G, \integer G) \cong \integer G \otimes_\integer \Ext^*_{\integer G}(\integer, \integer) \cong \integer G \otimes_\integer H^*(BG)$$
and therefore
the string topology ring for $[*/G]$ is the ring $\integer G \otimes_\integer H^*(BG)$.

\subsection{Manifold with finite fundamental group}
Consider a connected compact manifold $N$ with finite fundamental group $ G=\pi_1N$.
 Then by Theorem \ref{main theorem}
\begin{theorem}
The string topology ring $H_*(\LL N^{-TN};\integer)$ as defined in \citep{CohenJones} is
 isomorphic to the ring
$$HH^*(C^*(\tilde{N}) \# G, C^*(\tilde{N}) \# G)$$
where $\tilde{N}$ is the universal cover of $N$.
\end{theorem}
\begin{proof}
Because $G$ acts freely on $\tilde{N}$ we have that
\begin{align*}
H^{pro}_*\left(\LL(\tilde{N}\times_G EG)^{-T(\tilde{N} \times_G EG)}\right) & \cong H_*\left( P_G \tilde{N}^{-T\tilde{N}} \right)^G
\end{align*}
and moreover we have that as ring spectra $P_G \tilde{N}^{-T\tilde{N}} /G \cong \LL N^{-TN}$.
 Together with Theorem \ref{main theorem} we have the desired isomorphism of graded rings
$$H_*(\LL N^{-TN};\integer) \cong HH^*(C^*(\tilde{N}) \# G, C^*(\tilde{N}) \# G).$$
\end{proof}

\section{Failure of Hochschild cohomology invariance under orbifold equivalence} \label{Morita}

Taking a closer look at Theorem \ref{main theorem}, we see that we can only relate the string
topology of an orbifold to the Hochschild cohomology ring of the group dg-ring,
whenever we can write the orbifold as the groupoid $[M/G]$ with $M$ simply connected and connected.
The reason for this restriction lies in the use of the Eilenberg-Moore spectral sequence to relate the complex
$$C^*(M) \stackrel{L}{\otimes}_{C^*(M)^e} C^*(M) \ \ \ \  \mbox{with} \ \ \ \ C^*(\LL M),$$ which in the case that $M$ is not simply connected, it does not converge (see \citep{Dwyer}).

But it is natural to ask to whether the Hochschild cohomology ring
$$HH^*(C^*(M) \#G, C^*(M) \# G)$$
independent of the presentation of the orbifold; in other words, for two Morita equivalent  orbifold groupoids
$[M/G]$ and $[N/H]$ in the sense of \citep{Moerdijk}, are the graded rings
$$HH^*(C^*(M) \#G, C^*(M) \# G) \ \ \ \ \mbox{and} \ \ \ \ HH^*(C^*(N) \#H, C^*(N) \# H)$$
isomorphic?

The following result tells us that in general the answer of the previous question is negative.

\begin{proposition}\label{non-invarians S^1}
Let $M=S^1$ be the circle equipped with the antipodal action of
$G=\integer/2$. Let us replace the cochains by differential forms
$\Omega M$ and let us work complex coefficients.
Consider the Morita equivalent groupoids $[M/G]$ and the topological quotient $M/G$,
together with the associated dg-rings
$\Omega M \# G$ and $\Omega (M/G)$.
Then $HH^*(\Omega
M \# G, \Omega M \# G)$ is not isomorphic to $HH^*(\Omega(M/G), \Omega (M/G))$ as rings.
\end{proposition}

\begin{proof}  Let $\mathbb{C}[x]/x^2$ be the dg-algebra with trivial
differential and $\deg x = 1$. Consider the trivial $G$-action on
$\mathbb{C}[x]/x^2$.  Note that $H^*(\Omega M)$ is spanned by the
classes of $1$ and $d\Theta$. Since the elements $1$ and $d\Theta$
are $G$-invariant we get a $G$-equivariant quasi-isomorphism
\begin{equation}\label{q i omega m}
\mathbb{C}[x]/x^2 \to \Omega M, \; x \mapsto d\Theta
\end{equation}
Since $M/G$ is diffeomorphic to $M$ a straightforward computation
now gives
$$
HH^*(\Omega (M/G)) \cong HH^*(\Omega M) \cong
HH^*(\mathbb{C}[x]/x^2) \cong \mathbb{C}[y,z]/z^2
$$
where $\deg y = 0$ and $\deg z = 1$. On the other hand, the map in (\ref{q i
omega m}) gives a quasi-isomorphism
$$
\mathbb{C}[x]/x^2  \# G \cong \Omega M \# G.
$$
Since $\mathbb{C}[x]/x^2  \# G$ is the usual algebra tensor
product of $\mathbb{C}[x]/x^2$ and $\mathbb{C}G$ and since
$\mathbb{C} G$ is commutative and semi-simple this gives
$HH^*(\Omega M \# G) \cong$
$$
HH^*(\mathbb{C}[x]/x^2 \# G) = HH^*(\mathbb{C}[x]/x^2) \otimes
\mathbb{C} G \cong \mathbb{C}[y,z]/z^2 \otimes \mathbb{C}G.
$$
This proves the proposition since $\mathbb{C}[y,z]/z^2$ and
$\mathbb{C}[y,z]/z^2 \otimes \mathbb{C}G$ are non-isomorphic
rings.
\end{proof}
Since Hochschild cohomology commutes with extension of scalars and
$\Omega M$ is Morita equivalent to $C^* (M) \otimes_{\mathbb{Z}}
\mathbb{C}$, we conclude that $HH^*(C^* (M) \# G, C^*(M) \# G)$ and
$HH^*(C^*(M/G), C^*(M/G))$ are also non-isomorphic rings.

Because the Hochschild cohomology is invariant under derived equivalence (see \citep{Keller2}), we can conclude that
\begin{cor}
Let $M=S^1$ and endow it  with the antipodal action of $G=\integer/2$. Then the derived categories of dg-modules
$$\DD(C^*(M) \#G) \ \ \ \ \mbox{and} \ \ \ \ \DD(C^*(M/G))$$
(as in \citep{Schwede}) are not equivalent.
\end{cor}
(Notice that by Proposition \ref{non-invarians S^1} these
categories are non-equivalent also with coefficients in the field
of complex numbers.) Therefore one cannot expect that there is an
isomorphism of Hochschild cohomology rings for the dg-rings of
Morita equivalent groupoids, and in some sense, one can only
produce an isomorphism with the string topology ring whenever one
can find a description of the orbifold given by the quotient of a
simply connected manifold by the action of a finite group. When
this is not the case, as for example the case of manifolds with
non-finite fundamental group, we do not know how to recover the
string topology ring via the Hochschild cohomology of some dg-ring
associated to the orbifold. We leave this question open.

\section{Appendix A} \label{section derived}

\subsection{Derived category of dg-modules over a dg-ring}

In this section we will give the preliminaries on differential
graded modules over differential graded
 rings and we will setup the notation. We give a rather detailed exposition partly because
we have felt there was a need in the literature for an elementary
introduction to Hochschild (co)homology in the dg-setting
 that relates the derived category of dg-modules and the derived functor approach
 to the down to earth formulas used by the working topologists and algebraists.

 This summary is based on the papers \citep{Schwede, Keller} and the book
 \citep{BernsteinLunts}.
 In what follows all complexes will be cohomological i.e. the
 differentials will raise the degree by one.

 A {\bf differential graded ring} (dg-ring) is a pair $\Aa=(A,d)$ consisting of a  $\integer$-graded ring $A$
 together with a differential $d$ of degree $1$ which satisfies the
 Leibniz rules
 $$d(ab) = d(a)b + (-1)^{|a|}a d(b)$$
 for all homogeneous elements $a,b \in A$. In this paper we will
 assume that the dg-rings come endowed with a unit element. $A$ is called (graded) commutative if $ab =
 (-1)^{|a||b|}ba$.
 For the rest of this section $\Aa$ denotes a dg-ring.

 A {\bf differential graded left $\Aa$-module} (left $\Aa$-module) consists of a graded
 left $A$-module $M$ together with a
 differential $d_M$ of degree $1$ which satisfies the Leibniz rule
$$d_M(ab) = d(a)b + (-1)^{|a|}a  d_M(b)$$
 for all homogenous elements $a \in A$ and $b \in M$. A morphism
 $f:M \to N$ of $\Aa$-modules is an $A$-linear map which is homogenous of degree
  $0$ and commutes with the differentials.

Denote by $\Aa-\mathrm{mod}$ the category of left $\Aa$-modules.
We denote homomorphisms in this category by
$$ \Hom_{\Aa}(\_ \ ,\_ \ ) : = \Hom_{\Aa-{\mathrm{mod}}}(\_ \ ,\_\ ).$$
The category of graded left $A$-modules, is denoted by
$A-\mathrm{mod}$. Morphisms in this category are by definition
$A$-linear maps homogenous of degree $0$. We put
$$\Hom_A (\_ \ ,\_ \ ) : = \Hom_{A-{\mathrm{mod}}}(\_ \ ,\_\ ).$$
Thus
$$\Hom_\Aa(M,N) = \{f \in \Hom_A(M,N) | d_N \circ f - f \circ d_M
=0 \}.$$ Similarly, there are the categories ${\mathrm{mod}}-\Aa$
of right $\Aa$-modules and $\Aa-{\mathrm{mod}}-\Aa$ of
$\Aa$-bimodules.

The {\bf opposite} dg-ring $\Aa^o$ of $\Aa$ is defined to be
$\Aa^o=(A^o,d)$ where its elements are the same ones as in $\Aa$
and with the same differential but the multiplication $a \circ b$
is the opposite of the one in $A$, i.e.
$$a \circ b := (-1)^{|a||b|} ba.$$
Note that if $\Aa'$ is any dg-ring then $\Aa \otimes_\integer
\Aa'$ is a dg-ring with multiplication $a \otimes a'  \cdot b
\otimes b' = (-1)^{|a'||b|}ab \otimes a'b'$ and differential $d(a
\otimes a') = da \otimes a' +(-1)^{|a'|}a \otimes da'$, for $a,b
\in \Aa$ and $a',b' \in \Aa'$.

The {\bf shift functor} $[1]$ on $\Aa-{\mathrm{mod}}$ is given by
shifting the degree of a complex by one
$$M[1]^k= M^{1+k}$$
and with differential $d_{M[1]} = - d_M$. We therefore have a
canonical isomorphism of degree 1
\begin{eqnarray} \label{shift functor}
s : M \to M[1] \ \ \ \mbox{such that} \ \ \ \ \ \ \ s d_M x = -
d_{M[1]} sx.
\end{eqnarray}

The {\bf tensor product} of a right $\Aa$-module $M$ and a left
$\Aa$-module $N$ consists of the graded complex of abelian groups
$M \otimes_A N$ together with the differential
$$d(m \otimes n) = d_Mm \otimes n + (-1)^{|m|}m \otimes d_N n.$$
If $A$ is graded commutative then $M$ is automatically an
$\Aa$-bimodule by $amb = (-1)^{|m||b|}mab$, for $a,b \in A$ and $m
\in M$. Hence, in this case $M \otimes_{\Aa}N$ is a left
$\Aa$-module by $a\cdot m \otimes n = (am) \otimes n$.

A homomorphism $f : M \to N$ in $\Aa-{\mathrm{mod}}$ is a {\bf
quasi-isomorphism} if $f$ induces an isomorphism in cohomology
$H^*f : H^*M \cong H^*N$ .

A {\bf chain homotopy} between homomorphisms $f_0,f_1 \in
\Hom_\Aa(M,N)$ of $\Aa$-modules is a homomorphism of graded
$A$-modules  $k \in \Hom_A(M,N[-1])$ of degree $-1$ which
satisfies
$$f_1-f_0 = d_M \circ k + k \circ d_N.$$

The {\bf homotopy category} $\KK(\Aa)$ of the dg-ring $\Aa$ has
the same objects as $\Aa-{\mathrm{mod}}$ and as morphisms the
chain homotopy equivalence classes of morphisms in
$\Aa-{\mathrm{mod}}$.

The {\bf derived category} $\DD(\Aa)$ is the localization of the
homotopy category $\KK(\Aa)$ with respect to quasi-isomorphisms.

In more sophisticated presentations than ours one define a
structure of Quillen model category on $\Aa-\mathrm{mod}$, see
\citep{Hovey, Quillen-Book}. Although this topic is not discussed
here, we use below the model category term "cofibrant" for objects
which essentially do the same work in the dg-world (where objects,
e.g., the bar construction, automatically behave like unbounded
complexes) as bounded above projective resolutions do in classical
homological algebra.

A (left $\Aa$-)module $M$ is {\bf cofibrant with respect to $\Aa$} if there exists an
exhaustive increasing filtration by submodules
$$0=M^0 \subset M^1 \subset M^2 \subset \cdots \subset M^k \subset \cdots \subset M$$
such that each subquotient $M^k/M^{k-1}$ is a direct summand of a
direct sum of shifted copies of $\Aa$. An $\Aa$-module $M$  is  {\bf cofibrant} if
 it is cofibrant with respect to both $\Aa$ and the base ring $\integer$. In the
 case that $\Aa$ were free over $\integer$, the cofibrant modules are the same as
  the cofibrant modules with respect to $\Aa$.

Up to chain homotopy
equivalence, the cofibrant modules are the ones that possesses
Keller's property (P) \citep{Keller}. Every quasi-isomorphism
between cofibrant modules is a chain homotopy equivalence and
every module can be approximated up to quasi-isomorphism by a
cofibrant module.

The derived category $\DD(\Aa)$ is equivalent to the full
subcategory of $\KK(\Aa)$ whose objects are cofibrant
$\Aa$-modules, see \citep{Keller}.

If $f : \Aa \to \BB$ is a quasi-isomorphism of dg-rings, then the
derived functors of restriction and extension of scalars induce
equivalences between the derived categories $\DD(\Aa)$ and
$\DD(\BB)$.

Let $M$ and $N$ be left $\Aa$-modules, the {\bf homomorphism
complex}
$$\HHom_\Aa(M,N)$$
between $M$ and $N$ is the complex defined as follows:
 in dimension $k \in \integer$ the chain group $\HHom_\Aa(M,N)^k$ is the group
 of graded $A$-modules homomorphisms of degree $k$, i.e.
 $$\HHom_\Aa(M,N)^k : = \Hom_A(M,N[k]),$$
and the differential
$$d: \HHom_\Aa(M,N)^k \to \HHom_\Aa(M,N)^{k+1}$$
is defined by
$$d(f):=d_N \circ f - (-1)^k f \circ d_M.$$

With this definition in mind, the 0-cycles of the homomorphism
complex $\HHom_\Aa(M,N)$ are precisely the $\Aa$-module
homomorphisms between $M$ and $N$:
$$\ZZ^0 \HHom_\Aa(M,N) = \Hom_\Aa(M,N)$$
and the zero-th cohomology of the complex $\HHom_\Aa(M,N)$ is
precisely the set of equivalence classes of chain homotopic maps:
$$H^0 \HHom_\Aa(M,N) = \Hom_{\KK(\Aa)}(M,N).$$

For $M,N$ and $L$ left $\Aa$-modules composition of maps induces a
bilinear pairing
\begin{eqnarray} \label{dg-ring structure Hom_A(M,M)} \HHom_\Aa(N,L)^k \times \HHom_\Aa(M,N)^l \to \HHom_\Aa(M,L)^{k+l}
\end{eqnarray}
that moreover satisfies the Leibniz rule: for $f \in
\HHom_\Aa(M,N)^l$ and $g \in \HHom_\Aa(N,L)^k$ one can check that
$$d(g \circ f) = dg \circ f + (-1)^k g \circ df.$$

From this it follows in particular that the endomorphism complex
$\HHom_\Aa(M,M)$ is a differential graded ring with multiplication
the composition of homomorphisms.

\subsection{Derived functors}

For $M$ a left $\Aa$-module and $N$ a right $\Aa$-module, the
derived tensor $\stackrel{L}{\otimes}_\Aa$ between $N$ and $M$  is
defined as the complex
$$N \stackrel{L}{\otimes}_\Aa M := N \otimes_\Aa M'$$
where $M' \to M$ is a cofibrant replacement for $M$. The derived
tensor product defines a functor $\stackrel{L}{\otimes}_\Aa:
\,\DD(\Aa^o) \times \DD(\Aa) \To \DD(\integer)$.

The $\Tor$-groups between $N$ and $M$ are defined as the
cohomology of the derived tensor between $N$ and $M$,
$$\Tor^*_\Aa(N,M) = H^* (N \stackrel{L}{\otimes}_\Aa M).$$

Let $M$ and $N$ be two $\Aa$-modules, the {derived functor} of
$\HHom_\Aa$ is  the functor $\RHHom_\Aa$ which is defined as the
complex
$$\RHHom_\Aa(M,N) := \HHom_\Aa(M',N)$$
where $M' \to M$ is a cofibrant replacement of $M$. We see that
$\RHHom$ is well defined up to chain homotopies and it defines a
functor $\RHHom_\Aa(\;,\;): \,\DD(\Aa) \times \DD(\Aa) \To
\DD(\integer)$.

By definition we have
$$\Ext^k_\Aa(M,N) := \Hom_{\DD(\Aa)}(M,N[k])=H^k \RHHom_\Aa(M,N)=$$
$$H^k \HHom_\Aa(M',N).$$

We have seen that the endomorphism complex
$$\HHom_\Aa(M',M')$$ becomes a differential graded ring by composition of homomorphisms. This
in particular implies that the graded $\Ext$-group
\begin{equation*}
\Ext^*_\Aa(M,M)=H^* \HHom_\Aa(M',M')
\end{equation*}
becomes a graded ring.

\subsection{Hochschild (co)homology for dg-rings}

In this section we define the Hochschild homology and cohomology
for a dg-ring $\Aa$ and we will list some of its properties.

Consider the dg-ring $$\Aa^e:= \Aa \otimes_\integer \Aa^o.$$
Then an $\Aa$-bimodule is the same thing as a left $\Aa^e$-module.
Let us consider $\Aa$ as a left $\Aa^e$-module in the natural way.

In order to define the Hochschild (co)homology we need to define a
 replacement of $\Aa$ in $\Aa^e-\mathrm{mod}$
 that it is known as the Bar construction.

\subsubsection{Bar construction} \label{section bar}

The Bar construction is based on the Bar resolution for modules
over rings. We will use the sign conventions defined in  \citep{FelixThomas, FelixHalperinThomas}
and in section \ref{Appendix B} we will show how these sign conventions arise.

 For $k \geq 0$ let $$P^{-k} := \left(\Aa \otimes
\Aa[1]^{\otimes k} \otimes \Aa \right) = \Aa^{\otimes k+2}[k]$$ be the $\Aa^e$-module defined in the
natural way by
$$(a \otimes b) (x_0|x_1| \dots |x_{k+1}) = (ax_0|x_1| \dots
|x_{k+1}b)$$
where $a \otimes b \in \Aa^e$, and $(x_0|x_1| \dots |x_{k+1})$ denotes the element
$$x_0 \otimes sx_1 \otimes \cdots \otimes sx_n \otimes x_{j+1}$$
in $P^{-k}$ where $sx$ denotes the image in $\Aa[1]$ of $x \in \Aa$ under the isomorphism
$s: \Aa \to \Aa[1]$ induced by the shift functor; see (\ref{shift functor}).

We have that the degree of an element in $P^{-k}$ is
\begin{eqnarray*}
|(a_0| \dots |a_{k+1})|:= |a_0 | + \dots + |a_{k+1}| - k.
\end{eqnarray*}
and the differential in $P^{-k}$ becomes
\begin{eqnarray} \nonumber
d(x_0|x_1| \dots |x_{k+1}) & :=  & (dx_0|x_1|\dots |x_{k+1}) - \sum^{k}_{j=1}(-1)^{\varepsilon_{j-1} }(x_0|\dots |dx_j|\dots |x_{k+1})\\
\label {internal differential Aa^k}& & \ \ \ \ \ \ \ \ \ \ \ \ \ \   \ \ \ \ \ \ \ \ \ \ \ \ \ \ \ \ \  + (-1)^{\varepsilon_{k}}(x_0|\dots|x_k|dx_{k+1})
\end{eqnarray}
where
\begin{eqnarray*}
\varepsilon_j:= |x_0 | + |x_1| + \cdots +|x_j| -j
\end{eqnarray*}
denotes the degree of the first $j+1$ elements of $(a_0| \dots | a_{k+1})$ as an element in $\Aa \otimes \Aa[1]^k \otimes \Aa$.

Define the homomorphisms
of $\Aa^e$-modules
\begin{eqnarray*} \label{differential delta}
\delta^{-k} : P^{-k} & \to  & P^{-k+1}\\
(x_0|x_1| \dots |x_{k+1}) & \mapsto & \sum_{j=0}^{k-1} (-1)^{\varepsilon_j} (x_0|
\dots |x_{j-1}|x_jx_{j+1}|x_{j+2}| \dots  |x_{k+1}) \\
& & \ \ \ \ \ \ \ \ \ \ \ \ \ \ \ - (-1)^{\varepsilon_k} (x_0|
\dots |x_{k-1}|x_{k}x_{k+1})
\end{eqnarray*}
which together with the module homomorphism
\begin{eqnarray*}\label{epsilon def}
\epsilon: P^{0} & \to  & \Aa\\
(x_0|x_1) & \mapsto & x_0x_1
\end{eqnarray*}
determines a complex of $\Aa^e$-modules over
$$ \cdots P^{-3} \stackrel{\delta^{-3}}{\To} P^{-2} \stackrel{\delta^{-2}}{\To}
 P^{-1} \stackrel{\delta^{-1}}{\To} P^0 \stackrel{\epsilon}{\To} \Aa \To 0$$
which turns out to be acyclic if we consider it as a complex of
modules over $A^e$.

The bar construction is the $\Aa^e$-module
$$B(\Aa) := \bigoplus_{k=0}^{\infty} P^{-k} = \bigoplus_{k=0}^{\infty} \left(\Aa \otimes \Aa[1]^{\otimes k} \otimes \Aa \right)$$
with differential \begin{eqnarray*} \nonumber D : P^{-k}  & \to &
P^{-k}
\oplus P^{-k+1}\\
D(p) & = & d_{P^{-k}} p +  \delta_{-k}p
\end{eqnarray*}
that we will simply denote by $D=d + \delta$.

It is straightforward to check that the differentials $d$ and $\delta$ commute as operators, i.e. $[d,\delta] = d\delta - (-1)^{|d||\delta|} \delta d=0$, and therefore we have that $D^2=0$.

 The sign conventions for the differentials $d$ and $\delta$ are obtained by transporting the structure of the usual differentials $d$ and the Bar differential $\delta$ from $\oplus_k \Aa^{k+2}$ to $\oplus_k \left(\Aa \otimes \Aa[1]^k \otimes \Aa\right)$.
See section \ref{Appendix B} for a proof of this fact.

We extend the map $\epsilon$ of (\ref{epsilon def}) to a morphism
with $\epsilon: B(\Aa) \to \Aa$ (the augmentation morphism) by
requiring $\epsilon|_{P^{-k}}=0$ for $k
>0$.

\begin{lemma} \label{lemma bar is cofibrant}  Assume that the dg-ring $\Aa$ is free over
$\integer$. Then $B(\Aa) \stackrel{\epsilon}{\to} \Aa$ is a
cofibrant replacement of $\Aa$.
\end{lemma}

\subsubsection{Definition of Hochschild (co)homology}
With the bar construction at hand we can define the Hochschild
(co)homology groups.
\begin{definition}
The Hochschild cohomology ring of a dg-ring $\Aa$ is the ring
$$HH^*(\Aa,\Aa) := H^* \HHom_{\Aa^e}(B(\Aa),B(\Aa))$$
where the ring structure is given by composition of maps,
and the Hochschild homology groups is given by the graded group
$$HH_*(\Aa,\Aa):= H^* (B(\Aa) \otimes_{\Aa^e} \Aa).$$

In the case that $\Aa$ is free over $\integer$, we know from lemma \ref{lemma bar is cofibrant}
that $B(\Aa)$ is a cofibrant replacement for $\Aa$ in $\Aa^e$-mod, and therefore we could alternatively define
the Hochschild cohomology as the graded ring
$$HH^*(\Aa,\Aa):= \Ext^*_{\Aa^e}(\Aa,\Aa).$$
and the Hochschild homology
as the graded group
$$HH_*(\Aa,\Aa) := \Tor^*_{\Aa^e}(\Aa,\Aa).$$
\end{definition}

\subsubsection{Properties of the Hochschild cohomology}

 The Hochschild cohomology can be also calculated using the complex
$$\HHom_{\Aa^e}(B(\Aa),\Aa).$$
This complex is better suited for the  topological constructions that are done
in the rest of the paper in order to relate the homology of the free loops on a
 manifold and the Hochschild cohomology of the dg-ring of cochains.

We define the product $\phi \cdot \psi$ of  $\phi, \psi \in
\HHom_{\Aa^e}(B(\Aa), \Aa)$ by
\begin{eqnarray*} \label{product Hom(BA,A)}
(\phi \cdot \psi)(a_0| \dots | a_{k+1}) & = & \sum_{j=0}^k
(-1)^{|\psi|\varepsilon_j}\phi(a_0| \dots |a_j|1) \psi
(1|a_{j+1}|\dots |a_{k+1}).
\end{eqnarray*}

We shall show that this ring structure makes $\HHom_{\Aa^e}(B(\Aa),
\Aa)$ into a dg-ring,  which is moreover quasi-isomorphic to the dg-ring
$\HHom_{\Aa^e}(B(\Aa), B(\Aa))$ - where the dg-ring structure on
the latter complex is composition of maps; see (\ref{dg-ring
structure Hom_A(M,M)}).

\begin{proposition} $\HHom_{\Aa^e}(B(\Aa),
\Aa)$ is a dg-ring. \label{lemma Leibniz rule Hom(BA,A)}
\end{proposition}
\begin{proof}
The product is clearly associative, and moreover it defines the
structure of a unitary ring on $\HHom_{\Aa^e}(B(\Aa), \Aa)$; the
unit element of this ring is the augmentation map $\epsilon$.

The proof of the Leibniz rule is a calculation based on the fact that the ring
structure can also be obtained from the  diagonal map
$\Delta : B(\Aa) \to B(\Aa) \otimes_\Aa B(\Aa)$ of $\Aa^e$-modules  via  pullback.

The diagonal map $\Delta : B(\Aa)  \to  B(\Aa) \otimes_\Aa
B(\Aa)$ is defined by
\begin{eqnarray} \label{diagonal map bar}
\Delta(a_0| \dots | a_{k+1}) & = & \sum_{j=0}^k (a_0| \dots
|a_j|1) \otimes (1|a_{j+1}|\dots |a_{k+1})
\end{eqnarray}

For $\phi, \psi \in \HHom_{\Aa^e}(B(\Aa),\Aa)$ we see that
$$\phi \cdot \psi =  \phi \otimes \psi \circ \Delta = \Delta^*(\phi \otimes \psi),$$
where $\phi \otimes \psi: B(\Aa)\otimes_\Aa  B(\Aa) \to \Aa \otimes_\Aa \Aa=\Aa$ is
given by $a \otimes a' \mapsto (-1)^{|a||\psi|}\phi(a) \psi(a')$.
\end{proof}

Let us state some facts whose proofs are not difficult.

\begin{lemma} \label{cofibreplacementlemma} Assume that  $\PP \stackrel{\mu}{\to} \Aa$ is a cofibrant replacement of
$\Aa$ in $\Aa^e-\mathrm{mod}$. Then  $\PP \otimes_\Aa \PP
\stackrel{\mu \otimes \mu }{\to} \Aa \otimes_\Aa \Aa = \Aa$ is
also a cofibrant replacement of $\Aa$ in $\Aa^e-\mathrm{mod}$.
\end{lemma}

\begin{lemma}
The map $\Delta : B(\Aa) \to B(\Aa) \otimes_\Aa B(\Aa) $ is a map of $\Aa^e$-modules.
\end{lemma}

Notice that the compositions
$$B(\Aa) \stackrel{\Delta}{\To} B(\Aa) \otimes_\Aa B(\Aa)
\stackrel{1\otimes \epsilon}{\To} B(\Aa)\otimes_\Aa \Aa = B(\Aa)$$
$$B(\Aa) \stackrel{\Delta}{\To} B(\Aa) \otimes_\Aa B(\Aa)
\stackrel{ \epsilon \otimes 1}{\To} \Aa\otimes_\Aa B(\Aa) =
B(\Aa)$$ are the identity and that, by lemma
\ref{cofibreplacementlemma} $B(\Aa) \otimes_\Aa B(\Aa)
\stackrel{\epsilon \otimes \epsilon}{\To} \Aa \otimes_\Aa \Aa
=\Aa$ is a cofibrant replacement for $\Aa$ (in the category of
left $\Aa^e$-modules). Since the maps $1 \otimes \epsilon $ and
$(1 \otimes \epsilon) \circ \Delta$ are quasi-isomorphisms, also
$\Delta$ is a quasi-isomorphism.

Now note that since $\epsilon: B(\Aa) \to \Aa$ is a cofibrant
replacement it induces a quasi-isomorphism of complexes
\begin{equation*}
\epsilon_*: \HHom_{\Aa^e}(B(\Aa), B(\Aa)) \To \HHom_{\Aa^e}(B(\Aa),
\Aa),
\end{equation*}
but note that its kernel is not a right ideal and,
consequently, $\epsilon_*$ cannot be a ring homomorphism, nevertheless the properties of the diagonal map $\Delta$ gives us:

\begin{proposition}
The map on cohomology induced by $\epsilon_*$ is multiplicative,
hence gives a canonical ring isomorphism between
$$H^*\HHom_{\Aa^e}(B(\Aa), \Aa) \ \ \ \mbox{and} \ \ \ \
H^*\HHom_{\Aa^e}(B(\Aa),B(\Aa)).$$
\end{proposition}

To finish this section let us state that the construction performed with the diagonal map for $B(\Aa)$ can
be generalized to other cofibrant
replacements of $\Aa$:
\begin{proposition} \label{proposition diagonal map}
Let $\PP \stackrel{\mu}{\to} \Aa$ be a cofibrant replacement of
$\Aa$ and furthermore assume that there exist a homomorphism of
$\Aa^e$-modules
$$\Delta_\PP : \PP \to \PP \otimes_\Aa \PP$$
such that the compositions
\begin{eqnarray}
\PP \stackrel{\Delta_\PP}{\To} \PP \otimes_\Aa \PP \stackrel{1\otimes \mu}{\To} \Aa \otimes_\Aa \PP = \PP \label{counit PP}\\
\PP \stackrel{\Delta_\PP}{\To} \PP \otimes_\Aa \PP \stackrel{\mu \otimes 1}{\To} \PP \otimes_\Aa \Aa = \PP \nonumber
\end{eqnarray}
are both the identity, and that the map $\Delta_{\PP}$ is coassociative.
Then the product structure defined by the
map
\begin{eqnarray*}
\HHom_{\Aa^e}(\PP,\Aa) \times \HHom_{\Aa^e}(\PP,\Aa) & \to & \HHom_{\Aa^e}(\PP,\Aa)\\
\phi \times \psi & \mapsto & \Delta_\PP^* (\phi \otimes \psi)
\end{eqnarray*}
makes $\HHom_{\Aa^e}(\PP,\Aa)$ into an associative  dg-ring which induces an associative ring structure on $H^*\HHom_{\Aa^e}(\PP,\Aa)$.
This ring is canonically isomorphic to $HH^*(\Aa,\Aa)$.
\end{proposition}
\begin{proof}
The same argument as in lemma \ref{lemma Leibniz rule Hom(BA,A)} shows that
 $\HHom_{\Aa^e}(\PP,\Aa)$ is a dg-ring; its associativity follows from the
  coassociativity of $\Delta_\PP$. The fact that  $\mu$ is a unit follows from the
compositions of (\ref{counit PP}).

Now, since $B(\Aa)$ and $\PP$ are cofibrant, there exists a
quasi-isomorphism of $\Aa^e$-modules $\alpha : B(\Aa) \to \PP$,
unique up to homotopy such that $\mu \circ \alpha = \epsilon$.
This gives a map $\phi \mapsto \alpha^*(\phi)$ from
$ \HHom_{\Aa^e}(\PP,\Aa)$ to $\HHom_{\Aa^e}(B(\Aa),\Aa)$. Taking cohomology
we get a graded group isomorphism
\begin{equation}\label{sssasasas}
H^* \HHom_{\Aa^e}(\PP,\Aa) \to H^*\HHom_{\Aa^e}(B(\Aa),\Aa), \,\phi \mapsto
\alpha^*(\phi).
\end{equation}
which is independent of the choice of $\alpha$.

 Now consider the maps
$$\Delta_P \circ \alpha,
\alpha \otimes \alpha \circ \Delta: B(\Aa) \to \PP \otimes_{\Aa}
\PP.$$ We have that $\mu \otimes \mu : \PP \otimes_{\Aa} \PP \to
\Aa$ is a cofibrant replacement and that
$$\mu \otimes \mu \circ \Delta_P \circ \alpha = \mu \otimes \mu
\circ\alpha \otimes \alpha \circ \Delta = \epsilon \otimes
\epsilon.$$ We conclude that $\Delta_P \circ \alpha$ and $\alpha
\circ \Delta$ are homotopic. Hence they induce the same map on
cohomology which implies that (\ref{sssasasas}) is a ring
homomorphism. As $H^* \HHom_{\Aa^e}(B(\Aa),\Aa)$ is canonically isomorphic
to $HH^*(\Aa,\Aa)$, the proposition follows.

\end{proof}

\section{Appendix B} \label{Appendix B}

In this section we will explain how the sign conventions arise for the Bar construction of section \ref{section bar}.
The idea consists in transporting the standard differentials of the complex $\oplus_k\Aa^{k+2}$ to the complex $\oplus_k \left( \Aa \otimes \Aa[1]^k \otimes \Aa \right)$ via some chosen isomorphisms.

Let us recall from (\ref{shift functor}) that there is an isomorphism $s: \Aa \to \Aa[1]$ ; $a \mapsto sa$ such that $dsa= - sda$.
Define the isomorphism $$\Phi_k :\Aa^{k+2} \to \Aa \otimes \Aa[1]^k \otimes \Aa$$ as the composition of the maps
$$\left(\operatorname{Id}  \otimes s \otimes \operatorname{Id}^{\otimes k}\right) \circ \cdots \circ \left( \operatorname{Id} ^{\otimes k-1} \otimes s \otimes \operatorname{Id}^{\otimes 2}\right)  \circ \left( \operatorname{Id} ^{\otimes k} \otimes s \otimes \operatorname{Id}\right).$$

Because $s$ is an odd map we get
\begin{align*}\left( \operatorname{Id} ^{\otimes i} \otimes s \otimes \operatorname{Id}^{\otimes k-i+1}\right) (a_0|\dots | & a_{k+1}) =\\ &(-1)^{|a_0| + \cdots + |a_i|}
 (a_0|\dots | a_i | sa_{i+1} | a_{i+2}| \dots |a_{k+1})
 \end{align*}
 and therefore
 \begin{align*}
 \Phi_k(a_0 |\dots | a_{k+1}) = (-1)^{\sum_{i=0}^{k-1} (k-i)|a_i|} (a_0 | sa_1 | \dots |sa_k |a_{k+1}).
 \end{align*}

The standard differential for the Bar resolution is defined as the sum $\delta^0 + \dots + \delta^k$ with
\begin{align*}
\delta^j :\Aa^{k+2} & \to  \Aa^{k+1}\\
(a_0| \dots |a_{k+1}) & \mapsto  (-1)^j (a_0 | \dots |a_{j-1}| a_ja_{j+1} | a_{j+2} | \dots |a_{k+1}).
\end{align*}
We will transport these maps $\delta^j$ as maps $\Aa \otimes \Aa[1]^k \otimes \Aa \to \Aa \otimes \Aa[1]^{k-1} \otimes \Aa$ to define the differential defined in (\ref{differential delta}).

We have that
\begin{align*}
 \Phi_k   (a_0 | \dots & | a_ja_{j+1} |   \dots |a_{k+1})  = \\ & (-1)^{\sum_{i=0}^{j} (k-i-1)|a_i|  + \sum_{i=j+1}^{k-1} (k-i)|a_i|} (a_0 | sa_1|\dots | s(a_ja_{j+1}) | \dots |sa_k|a_{k+1})
\end{align*}
which implies that the induced map becomes
\begin{align*}
\bar{\delta}^j : \Aa \otimes \Aa[1]^k \otimes \Aa &  \to \Aa \otimes \Aa[1]^{k-1} \otimes \Aa \\
(a_0 |sa_1 | \dots | sa_k | a_{k+1}) & \mapsto (-1)^{\varepsilon_j} (a_0 | sa_1|\dots | s(a_ja_{j+1}) | \dots |sa_k|a_{k+1})
\end{align*}
which satisfies  $\bar{\delta}^j \circ \Phi_k = \Phi_{k-1} \circ \delta^j$ and with
$$\varepsilon_j := |a_0| + |a_1| + \cdots + |a_j| -j.$$
Note that when $j=k$ we get that
$$\bar{\delta}^k(a_0 |sa_1 | \dots | sa_k | a_{k+1}) = - (-1)^{\varepsilon_k} (a_0 | sa_1|\dots |sa_{k-1}|a_ka_{k+1})$$
and therefore the differential $\delta$ in the Bar construction $B(\Aa)$ becomes
$$\delta : = \bar{\delta}^0 + \dots +\bar{\delta}^k$$
which is the one defined in (\ref{differential delta}).

\smallskip

To finalize let us explain how the differential defined in (\ref{internal differential Aa^k}) comes from the internal differential of  $\Aa \otimes \Aa^k \otimes \Aa$ and the graded commutation of $s$ and $d$:
\begin{align*}
d(a_0 |  sa_1 & | \dots |   sa_k| a_{k+1}) \\ = & (da_0 |  sa_1 |
\dots | sa_k| a_{k+1}) +
 \sum_{j=1}^k (-1)^{\varepsilon_{j-1}} (a_0 | sa_1 | \dots|dsa_j| \dots  | sa_k| a_{k+1}) \\
& + (-1)^{\varepsilon_k} (a_0 | sa_1 |\dots  | sa_k| da_{k+1}) \\
 = &  (da_0 |  sa_1 | \dots | sa_k| a_{k+1}) -
 \sum_{j=1}^k (-1)^{\varepsilon_{j-1}} (a_0 | sa_1 | \dots|sda_j| \dots  | sa_k| a_{k+1}) \\
& + (-1)^{\varepsilon_k} (a_0 | sa_1 |\dots  | sa_k| da_{k+1}).
\end{align*}

\section{Appendix C} \label{Appendix C}
Let $M$ be a compact, oriented, differentiable, connected and simply connected manifold. In this section we
resolve the technical problems about the Hochschild cohomology of
the singular cochains $C^* := C^*(M)$ that arise from the fact
that $C^*$ is not a free $\integer$-module (unless $M$ is a
point). This is so because an infinite product of copies of
$\integer$ is not free over $\integer$. This problem would
disappear if one uses coefficients over a field, but we prefer to develop the theory over
$\integer$ as it will turn out that with some effort  it is
possible to do so. On the algebraic side we could solve this
problem by considering simplicial cochains $S^*$ instead of singular cochains (having an explicit simplicial decomposition of $M$). Note however that
the algebraic object that naturally corresponds to the topological
side of this paper is the Hochschild cohomology of $C^*$ and not
that of $S^*$, so we are forced to deal with $C^*$.

Recall that one has
$$HH^*(C^*, C^*) = H^*\HHom_{C^{*e}}(B(C^*),C^*)$$
and that
$$HH^*(S^*,S^*) = H^*\HHom_{S^{*e}}(B(S^*),S^*) \cong \Ext_{S^{*e}}(S^*,S^*).$$
The fact that $C^*(M)$ is not free over $\integer$ implies that
$B(C^*)$ cannot be cofibrant in the sense of section \ref{section derived}.
Nevertheless, this quantity can be interpreted as a
\emph{relative} Ext-group (see Lemma 9.1.3 in \citep{Weibel}) and it is even possible to
define a model category structure on the category of dg-modules
over $C^{*e}$ in which $B(C^*)$ is cofibrant. In principle, all
the homological algebra of this paper could be translated to this
setup, but it turns out that this is not necessary since we shall
prove that
\begin{theorem}\label{same HH} There is a canonical isomorphism $$HH^*(C^*,C^*) \cong HH^*(S^*,S^*)$$ as graded $\integer$-algebras.
\end{theorem}
The reason for this to be possible is that $C^*$ is the dual of
$C_*$ which is free over $\integer$. We have
\begin{lemma}\label{homotopy lemma} Let $(X,d)$ and $(Y, \partial)$ be (unbounded) complexes of free $\integer$-modules such that
$H^*(X) \cong H^*(Y)$. Then $X$ and $Y$ are homotopic and any
quasi-isomorphism $q: X \to Y$ is a homotopy equivalence.
\end{lemma}
\begin{proof}
For every $n$ we have that $X^n \to \operatorname{Im} d^n$ splits,
since $\operatorname{Im} d^n$ is a free $\integer$-module, being a
submodule of the free $\integer$-module $X^{n+1}$. Therefore $X^n
\cong \operatorname{Ker} d^{n} \oplus \operatorname{Im} d^n$ and
the differential $d^n:X^n \to X^{n+1}$ corresponds to $(a,b)
\mapsto (b,0)$.

Let $A_n$ denote the subcomplex $\operatorname{Im} d^{n-1}
\hookrightarrow \operatorname{Ker} d^{n}$ of $X$. It follows that
$$
X \cong \oplus_{n \in \integer} A_n
$$
Similarly, $Y \cong \oplus_{n \in \integer} B_n$, where $B_n$ is
the subcomplex $\operatorname{Im}
\partial^n \hookrightarrow \operatorname{Ker} \partial^{n+1}$ of
$Y$.

Since $H^n(A_n) \cong H^n(X) \cong H^n(Y) \cong H^n(B_n)$ and
$A_n$ is a projective resolution of $H^n(A_n)[-n]$ and $B_n$ is a
projective resolution of $H^n(B_n)[-n]$ we conclude that $A_n$ is
homotopic to $B_n$. Thus $X$ is homotopic to $Y$.

For the last assertion, note that each $A_n$ and $B_n$ are
cofibrant objects of $C(\integer)$. Thus $X$ and $Y$ are cofibrant
as well and hence any quasi-isomorphism between them is a homotopy.
\end{proof}

We fix a finite simplicial decomposition of $M$ on which $G$ acts
simplicialy and we consider the simplicial decomposition that the barycentric
subdivision defines (this is in order to get a simplicial decomposition which
induces a $G$-CW decomposition). If we think of a $k$-simplex in $M$ as a certain map
from $\Delta_k$ to $M$ it is clear how to get an embedding $i_*:
S_* \to C_*$ which is a $G$-equivariant quasi-isomorphism.
Moreover, if we denote by $f= i^*: C^* \to S^*$ the corresponding
map on cochains, then $f$ is a morphism of dg-algebras. We have
\begin{proposition} $f: C^* \to S^*$ is a homotopy equivalence.
\end{proposition}
\begin{proof} Since $S_*$ and $C_*$ are free $\integer$-modules, Lemma \ref{homotopy lemma}
 shows that $i_*$ is a homotopy-equivalence. Hence
also $f = i^*$ is a homotopy equivalence.
\end{proof}
We can now prove

\begin{proof}[of theorem \ref{same HH}] The dg-ring map $f: C^* \to S^*$
  induces a $C^{*e}$-module structure on $S^*$ and
a $B(C^*)^e$-linear map
$$
\overline{f}: B(C^*) \to B(S^*), \ (c_0 \vert \ldots \vert c_{k+1})
\mapsto (f(c_0) \vert \ldots \vert f(c_{k+1})).
$$

Note that $f$ will not have a multiplicative homotopy inverse and
hence that $\overline{f}$ is not a homotopy equivalence.

Nevertheless, we have canonical maps
\begin{equation*}
  \HHom_{C^{*e}}(B(C^*),C^*) \to
\HHom_{C^{e*}}(B(C^*),S^*), \
   \phi \mapsto  f \circ \phi
\end{equation*}
\begin{equation*}
  \HHom_{S^{*e}}(B(S^*),S^*) \to
\HHom_{C^{*e}}(B(C^*),S^*), \
   \phi \mapsto  \phi \circ \overline{f}
\end{equation*}
that we claim are both quasi-isomorphisms. This would prove Theorem \ref{same HH}.

 To show that each of the previous maps are q.i.'s, we will use specific spectral sequences
on both sides of each map, whose zeroth differentials will avoid the part of the total differential that
reflects the multiplication, and that will therefore become isomorphic at the first page. Let us be more explicit.

For $\Aa$ a dg-ring define $T(\Aa) = \oplus_{i \geq 0}
\Aa^{\otimes i}[i]$. Consider the restriction isomorphisms
\begin{align*}
\HHom_{\Aa^e}(\Aa^{\otimes i +2}[i], \Aa) \to &
\HHom_{\integer}(\Aa^{\otimes i}[i], \Aa) \\
 f \mapsto  & \{\mathbf{a}
\mapsto f(1 \otimes \mathbf{a} \otimes 1)\}
\end{align*}
whose inverse isomorphism is $$g \mapsto \{(a_0 \vert \ldots \vert
a_{i+1}) \mapsto a_0g(a_1 \vert \ldots \vert a_i)a_{i+1}\}.$$

If we use these isomorphisms to transport the differential  in
$\HHom_{\Aa^e}(B(\Aa),\Aa)$ to $\HHom_\integer(T(\Aa),\Aa)$ we get
an isomorphisms of these two homomorphism complexes. Explicitly,
the differential $d$ on $\HHom_\integer(T(\Aa),\Aa)$ is given by
$$
df(\mathbf{a}) = d_t(f(\mathbf{a})) - (-1)^{|f|} (\delta f (\mathbf{a}) + f(d_s(\mathbf{a})))
$$
where $d_t$ is the differential of $\Aa$, the target of the homomorphism,   $d_s$ is the differential on the source, which becomes
\begin{align*}
d_s(a_1|...|a_k) = -\sum_{i=1}^{k} (-1)^{\varepsilon_{i-1}} (a_1 | ... |da_i | ... | a_k),
\end{align*}
and $\delta f$ is the homomorphism defined as
\begin{align*}
\delta f (a_1 | ... | a_k) = & a_1 f(a_2|...|a_k) + \sum_{i=2}^{k-1}(-1)^{\varepsilon_i}f(a_1 |...|a_ia_{i+1}|...|a_k) \\
 & - (-1)^{\varepsilon_{k}}f(a_1|...|a_{k-1})a_k.
\end{align*}

Applying the previous discussion to $\Aa = C^*$ and $\Aa = S^*$ we see
that it suffices to show that the following two dg-maps are
quasi-isomorphisms:
\begin{equation*}\label{first q.i.}
  \HHom_{\integer}(T(C^*),C^*) \to
\HHom_{\integer}(T(C^*),S^*), \
   \phi \mapsto  f \circ \phi
\end{equation*}
\begin{equation*}\label{second q.i.}
  \HHom_\integer(T(S^*),S^*) \to
\HHom_\integer(T(C^*),S^*), \
   \phi \mapsto  \phi \circ \overline{f}
\end{equation*}

Let us prove that the first map is a quasi-isomorphism

\begin{lemma} \label{lemma first q.i.}
The map
\begin{equation*}
 \Phi : \HHom_{\integer}(T(C^*),C^*) \to
\HHom_{\integer}(T(C^*),S^*), \
   \phi \mapsto  f \circ \phi
\end{equation*}
is a quasi-isomorphism.
\end{lemma}
\begin{proof}
Consider the filtrations
\begin{eqnarray*}
F_k = \HHom_\integer \left( \bigoplus_{i=k}^\infty C^{* \otimes i}[i], C^*\right) \ \ \ \ \ \
\overline{F}_k = \HHom_\integer \left( \bigoplus_{i=k}^\infty C^{* \otimes i}[i], S^*\right)
\end{eqnarray*}
of $ \HHom_{\integer}(T(C^*),C^*)$ and $\HHom_{\integer}(T(C^*),S^*)$ respectively,
 which are compatible with the homomorphism $\Phi$.

The associated spectral sequences associated to $F_*$ and $\overline{F}_*$ have for zeroth page the complexes
\begin{eqnarray*}
E_0 = \bigoplus_{k=0}^\infty  F_k/F_{k+1}= \bigoplus_{k=0}^\infty  \HHom_\integer( C^{* \otimes k}[k],C^*)\\
\overline{E}_0 = \bigoplus_{k=0}^\infty  \overline{F}_k/\overline{F}_{k+1}= \bigoplus_{k=0}^\infty  \HHom_\integer( C^{*\otimes k}[k],S^*)
\end{eqnarray*}
and whose zeroth differentials are
\begin{eqnarray*}
d^0 \phi ({\bf a}) = d_{C^*}(\phi({\bf a})) - (-1)^{|\phi|} \phi(d_s({\bf a}))  \\ \overline{d}^0 \overline{\phi} ({\bf a}) = d_{S^*}(\overline{\phi}({\bf a})) - (-1)^{|\overline{\phi}|} \overline{\phi}(d_s({\bf a})).
\end{eqnarray*}
where $d_s$ is the internal differential of the complex $C^{* \otimes k}[k]$.

Applying the canonical isomorphisms
\begin{eqnarray*}
\HHom_\integer(C^{* \otimes k}[k], C^*) \cong \HHom_\integer(C^{* \otimes k}[k]\otimes C_*, \integer)\\
\HHom_\integer(C^{* \otimes k}[k], S^*) \cong \HHom_\integer(C^{* \otimes k}[k]\otimes S_*, \integer)
\end{eqnarray*}
we see that the induced differentials on the right hand side, are just the canonical differentials
 given by the dualization of the tensor product of $d_s$ with $\partial_{C_*}$
 and of $d_s$ and $\partial_{S_*}$ respectively, where $\partial$ denotes the differential at the chains level.

Because the map $i: S_* \to C_*$ induces a homotopy equivalence, we see that the induced map on the
 zeroth pages $\Phi: E_0 \to \overline{E}_0$ is a quasi-isomorphism, and therefore the induced
 homomorphism on the first pages becomes an isomorphism $\Phi: E_1 \stackrel{\cong}{\to} \overline{E}_1.$
We have therefore that the spectral sequences are isomorphic after the first page and moreover that
 the filtrations are both Hausdorff and weakly convergent.

The convergence of the spectral sequences tell us that $\Phi$ induces an isomorphism between
the inverse limits (see Theorem 3.9 of \citep{McCleary})
\begin{align} \nonumber
\Phi: \lim_{\leftarrow k} \left( H^*\HHom_\integer(T(C^*),C^*) / {\rm Im}(H^*F_k \to H^*\HHom_\integer(T(C^*),C^*))\right) & \\
\stackrel{\cong}{\to} \lim_{\leftarrow k} \left( H^*\HHom_\integer(T(C^*),S^*) / {\rm Im}(H^*\overline{F}_k \to  H^*\HHom_\integer(T (C^*) \right. & \left.  ,  S^*))\right) \label{iso completions}
\end{align}
but as the cohomology groups $H^*\HHom_\integer(T(C^*),C^*)$ and $H^*\HHom_\integer(T(C^*),S^*)$ are both graded groups, and in each degree they are finitely generated (as we know that $H^*\HHom_\integer(T(C^*),C^*)$ is isomorphic to the homology of $\LL M^{-TM}$), then we have that (\ref{iso completions}) implies that there is $\Phi$ induce an isomorphism at the level of the cohomologies
$$\Phi : H^*\HHom_\integer(T(C^*),C^*) \stackrel{\cong}{\to}  H^*\HHom_\integer(T(C^*),S^*).$$
This finishes the proof of Lemma \ref{lemma first q.i.}.
\end{proof}

Let us prove that the second map is a quasi-isomorphism

\begin{lemma}
The map
\begin{equation*}
 F : \HHom_{\integer}(T(S^*),S^*) \to
\HHom_{\integer}(T(C^*),S^*), \
   \phi \mapsto  \phi \circ \overline{f}
\end{equation*}
is a quasi-isomorphism.
\end{lemma}

\begin{proof}
Consider the following double filtrations
 \begin{align*}P^{j.k} = \HHom_\integer( \bigoplus_{i \geq k} S^{*\otimes i} , S^{\geq j})\\
 \overline{P}^{j.k} = \HHom_\integer( \bigoplus_{i \geq k} C^{*\otimes i} , S^{\geq j})
 \end{align*}
  and note that the filtrations are compatible with the map $F$, and that the differentials $\delta$ and $d_t=d_{S^*}$ raise the degree, namely we have that
\begin{align*}
 \delta : P^{j,k} \to P^{j,k+1},   & d_t: P^{j,k} \to P^{j+1,k},  \ \  \delta : \overline{P}^{j,k} \to \overline{P}^{j,k+1}  \\  \mbox{and}  \  & d_t: \overline{P}^{j,k} \to \overline{P}^{j+1,k}.
\end{align*}

Let us take the filtrations defined by these double filtrations
 $$Q^r= \sum_{k+j=r} P^{j,k} \ \ \ \ \ \overline{Q}^r= \sum_{k+j=r} \overline{P}^{j,k}$$
 noting that both $\delta$ and $d_t$ raise the degree $\delta, d_t: Q^r \to Q^{r+1}$.
 Therefore the associated spectral sequences $E_*$ and $\overline{E}_*$ are compatible via the map $F$
 $$F: E_* \to \overline{E}_*$$
 and on the zeroth pages of both spectral sequences we get the associated graded
 $$E_0 \cong \bigoplus_{r=0}^\infty \bigoplus_{k+j=r} \HHom_\integer( S^{*\otimes k}, S^j)$$
 $$\overline{E}_0 \cong \bigoplus_{r=0}^\infty \bigoplus_{k+j=r} \HHom_\integer( C^{*\otimes k}, S^j).$$

The zeroth differential $d^0$ on the group $\HHom_\integer( S^{* \otimes k}, S^j)$ becomes the differential obtained by pre-composing with the internal differential of the source $S^{*\otimes k}$
 $$(d^0\phi )(a_1|...|a_k)= (-1)^{|\phi|} \phi(d(a_0|...|a_k))$$
 (the same happens with the zeroth differential $\overline{d}^0$ of $\overline{E}_0$), and therefore we have that
 the first level of the spectral sequences become the sum of the cohomologies of the duals of $S^{\otimes k}$ and $C^{\otimes k}$ respectively tensored with $S^*$
 $$ E_1  \cong \bigoplus_{r=0}^\infty  H^*( (S^{* \otimes r})^\vee) \otimes S^*$$
 $$ \overline{E}_1  \cong \bigoplus_{r=0}^\infty  H^*( (C^{* \otimes r})^\vee) \otimes S^*$$
 as we know that $S^*$ is a finitely generated free $\integer$-module.

 The map $F:E_1 \to \overline{E}_1$ at the first level is clearly an isomorphism as the map $f: C^* \to S^*$ induces a quasi-isomorphism at the level of the tensor products $C^{* \otimes k} \stackrel{\simeq}{\to} S^{* \otimes k}$ and their duals $(S^{* \otimes k})^\vee \stackrel{\simeq}{\to} (C^{* \otimes k})^\vee$.

 We have now that $F$ becomes an isomorphism at the first level of the spectral sequences.
 We also have that the filtration given by $Q^r$ is Hausdorff; $\cap_r Q^r=\{0\}$, therefore weakly-convergent, and applying the same argument as in the previous lemma, namely that the cohomologies
 $$H^*\HHom_\integer(T(S^*),S^*) \ \ \ \ \mbox{and} \ \ \ \ H^*\HHom_\integer(T(C^*),S^*)$$
 are both graded rings, we have that the isomorphism of the spectral sequences implies that the cohomologies
 $$H^*\HHom_\integer(T(S^*),S^*) \ \ \ \ \mbox{and} \ \ \ \ H^*\HHom_\integer(T(C^*),S^*)$$
 are isomorphic. Therefore  $F$ induces a q.i. on the complexes
 $$\HHom_\integer(T(S^*),S^*) \stackrel{\simeq}{\to} \HHom_\integer(T(C^*),S^*).$$
\end{proof}
This ends the proof of Theorem \ref{same HH}.
\end{proof}

Using the notation of (\ref{definition Aa_h}) we see that Theorem \ref{same HH} implies that there are canonical isomorphisms
\begin{align*}
H^*\HHom_{C^{*e}}(B(C^*),C^*_g) \cong &  H^*\HHom_{S^{*e}}(B(S^*),S^*_g) = \Ext_{S^{*e}}(S^*,S^*_g)\\
H^*\HHom_{C^{*e}}(B(C^*),C^*\#G) \cong &  H^*\HHom_{S^{*e}}(B(S^*),S^*\#G)= \Ext_{S^{*e}}(S^*,S^*\#G) \nonumber
\end{align*}
which induces a canonical isomorphism between the cohomologies
\begin{align*}
H^*\HHom_{\integer G}(\overline{B}(\integer G),\HHom_{C^{*e}}(B(C^*),C^* \# G)) \cong & \\ H^*\HHom_{\integer G}(\overline{B}(\integer G),\HHom_{S^{*e}} & (B(S^*),S^*\#G)).
\end{align*}
Therefore we can conclude that
\begin{cor}
There is a canonical ring isomorphism
$$HH^*(C^* \#G, C^* \#G) \cong HH^*(S^* \#G, S^* \#G)$$
which can also be seen as a ring isomorphism
$$HH^*(C^* \#G, C^* \#G) \cong \Ext_{S^{*}\#G^e}(S^* \# G, S^* \# G).$$
\end{cor}

\bibliographystyle{plainnat}
\bibliography{Hochschild}

\affiliationone{
  A. \'Angel, E. Backelin and B. Uribe\\
   Departamento de Matem\'{a}ticas\\ Universidad de los Andes\\
Carrera 1 N. 18A - 10\\ Bogot\'a \\
COLOMBIA
   \email{ja.angel908@uniandes.edu.co\\
   erbackel@uniandes.edu.co\\
   buribe@uniandes.edu.co}}

\end{document}